\documentclass{amsart}

\usepackage{amssymb}
\usepackage{color}
\usepackage{amsxtra} 
\usepackage{mathrsfs} 
\usepackage{yfonts} 
\usepackage{graphics}
\usepackage{ulem}
\numberwithin{equation}{section}

\newtheorem{theorem}{Theorem}[section]

\newtheorem{lem}{Lemma}

\newtheorem{prop}[theorem]{Proposition}
 
\theoremstyle{definition}
\newtheorem{definition}{Definition}

\newtheorem{cor}[theorem]{Corollary} 
\theoremstyle{remark}
\newtheorem{remark}{Remark}
\newtheorem{fact}[theorem]{Fact}

\newcommand\DistTo{\xrightarrow{
   \,\smash{\raisebox{-0.65ex}{\ensuremath{\scriptstyle\sim}}}\,}}
   
   \makeatletter
\renewcommand*{\eqref}[1]{%
  \hyperref[{#1}]{\textup{\tagform@{\ref*{#1}}}}%
}
\makeatother

\allowdisplaybreaks[3]

\begin{document}

\title[conjugate point in the Euler flow]
{Existence of a conjugate point in the incompressible Euler flow on an 
ellipsoid} 

\author{Taito Tauchi}
\address{Institute of Mathematics for Industry, Kyushu University, 744 Motooka, Nishi-ku, Fukuoka 819-0395, Japan} 
\email{tauchi.taito.342.@m.kyushu-u.ac.jp} 

\author{Tsuyoshi Yoneda} 
\address{Graduate School of Mathematical Sciences, University of Tokyo, Komaba 3-8-1 Meguro, 
Tokyo 153-8914, Japan } 
\email{yoneda@ms.u-tokyo.ac.jp} 


\date{\today} 



\begin{abstract} 
Existence of a conjugate point in the incompressible Euler flow on a sphere and an ellipsoid 
is considered.  Misio{\l}ek (1996) formulated a differential-geometric criterion (we call the M-criterion)
for the existence of a conjugate point in a fluid flow.
In this paper, it is shown that no zonal flow  (stationary Euler flow) satisfies
the M-criterion if the background manifold is a sphere, on the other hand,
there are zonal flows satisfy the M-criterion if the background manifold is an ellipsoid (even it is sufficiently close to the sphere).
The conjugate point is created by the fully nonlinear effect of the inviscid fluid flow
with differential geometric mechanism.

\end{abstract} 

\maketitle

\section{Introduction} 
\label{sec:Intro} 
A remarkable
example of a stable multiple zonal jet flow can be observed in 
Jupiter even against the
perturbation by,
for example,
the famous Great Red Spot. 
Despite attracting considerable
attention over the years, 
its mechanism is not yet well understood.
The incompressible 2D-Navier-Stokes equations on a rotating sphere 
are one of the simplest models of it, and many researchers have been extensively studying 
this models. Williams \cite{W} was the first to find that 
turbulent flow becomes multiple jet flows on such a model. 
However, he was assuming high symmetry to the flow field.
After that Yoden-Yamada \cite{YY} and Nozawa-Yoden \cite{NY} made further progress.
In particular, Obuse-Takehiro-Yamada \cite{O} calculated non-forced 2D-Navier-Stokes flow (without the
symmetry assumption) on a rotating sphere, and observed multiple zonal jet flows merging with each other and finally, only two or three broad zonal jets remain. 
Thus, it seems
that we need to find a totally different idea to clarify the existence of stable multiple zonal jet flow in Jupiter.
For the recent development in this study field, see Sasaki-Takehiro-Yamada \cite{STY-1,STY-2}, using a spectral method to the linearized fluid equations.

However, as far as the authors are aware, none
of the numerous works to date attempted to investigate the effect of the ``background manifold'' itself. In the above simplest model, the background manifold is a ``sphere", even though in reality
Jupiter is not a sphere. It has a perceptible bulge around its equatorial middle and is flattened at the poles (see \cite{web}). 
In this paper, we investigate the
effect of the background manifold, in particular, clarify the crucial difference between sphere and ellipsoid. 
 Let us explain more precisely.
Misio{\l}ek \cite{Mstability} showed Lagrangian instability of the stationary Euler flow with zero pressure term
on a manifold with non-positive curvature. 
He proved it using differential geometric
techniques based on Jacobi fields analysis. From the pioneering work of
Arnold (see \cite{A,A-P,AK,Smolentsev} for example) it is known that solutions to the incompressible Euler equations can
be seen as geodesics on the configuration space of diffeomorphisms of the
background manifold. Furthermore, negative curvature of the configuration
space as well as absence of conjugate points along these geodesics can be
regarded as suggesting Lagrangian instability of the corresponding fluid
flows. 
In this study, we will thus view existence
of a conjugate point 
as 
a suggestion of Lagrangian stability.
More precisely,
the existence of conjugate points should 
imply geodesics are certainly 
less unstable 
and
the strong positivity of 
the sectional curvature of the 
configuration space.
See Definition \ref{definition of conjugate point}
for the definition of the conjugate point
and
see also Nakamura-Hattori-Kambe \cite{NHK} for the explanation of Lagrangian instability.
(There exists an approach using numerical simulations
to
an Euler-Lagrangian analysis of the Navier-Stokes equations,
for example \cite{Okitani},
in which
the author
considered the time evolution of the sectional curvatures
and some solutions to the imcompressible Euler
equations.)
Subsequently, Misiolek \cite{MconjT2}  formulated a geometric criterion (we call the M-criterion, see \eqref{first-M-criterion} and \eqref{eq-def-M-criterion})
which is
a sufficient condition for the existence of a conjugate point in a fluid
flow.
Moreover,
he also showed 
that
there exists a conjugate point along a geodesic 
of the diffeomorphism group 
$\mathcal D^s_\mu(\mathbb{T}^2)$ of
the 2-dimensional flat torus ${\mathbb T}^{2}$.
Note that the conjugate point is created by the fully nonlinear effect of the inviscid fluid flow with differential geometric mechanism.
In this paper, we show that no zonal flow (a stationary Euler flow)
satisfies the M-criterion if the background manifold is a sphere but that
some zonal flows satisfy the M-criterion if the background manifold is an ellipsoid (even it is sufficiently close to the sphere), in particular, 
having a bulge around its equatorial middle and is flattened at the poles.

For the precise statement of our main theorems,
we briefly recall the theory of ``diffeomorhphism groups''
in the context of 
inviscid fluid flows
and the M-criterion.
See Section \ref{preliminary} for the details.

Let $(M,g)$ be a compact $n$-dimensional Riemannian manifold
without boundary.
Write
${\mathcal D}^{s}(M)$ for
the group of Sobolev $H^{s}$ diffeomorphisms of $M$
and
${\mathcal D}_{\mu}^{s}(M)$ for
the subgroup of ${\mathcal D}^{s}(M)$
consisting volume preserving elements,
where $\mu$ is the volume form on $M$ defined by $g$.
If $s>\frac{n}{2}+1$,
the group ${\mathcal D}^{s}(M)$
can be 
given a structure of an infinite-dimensional weak Riemannian manifold
(see \cite{EMa})
and ${\mathcal D}^{s}_{\mu}(M)$ is its weak Riemannian submanifold.
This
weak
Riemannian metric on ${\mathcal D}^{s}_{\mu}(M)$
is given
by
\begin{eqnarray}
\left(V, W\right)
:=
\int_{M}
g(V,W)\mu,
\label{metric}
\end{eqnarray}
where
$V,W\in T_{\eta}{\mathcal D}^{s}_{\mu}(M)$.
Here,
we identify
the tangent space $T_{\eta}{\mathcal D}^{s}_{\mu}(M)$ 
of ${\mathcal D}^{s}_{\mu}(M)$ at a 
point $\eta\in{\mathcal D}^{s}(M)$ with
the space of
all 
$H^{s}$ 
divergence-free 
sections of the pullback bundle $\eta^{*}TM$
of the tangent bundle $TM$.
Then,
if
$\eta(t)$ is
a geodesic with respect to this metric in ${\mathcal D}^{s}_{\mu}(M)$
joining $e$ and $\eta(t_{0})$,
a time dependent vector field on $M$ defined by
$u(t):=\dot{\eta}(t)\circ \eta^{-1}(t)$ 
is a solution to the Euler equations on $M$:
\begin{align} \nonumber 
&\partial_tu + \nabla_{u} u = - {\rm grad}\:p
\qquad 
t \in [0,t_{0}], 
\\ 
\label{Eintro}
&\mathrm{div}\, u=0,
\\ 
\nonumber 
&u|_{t=0} = \dot{\eta}(0), 
\end{align} 
with a scalar  function (pressure) $p(t)$ determined by $u(t)$.
In this context,
the existence of conjugate points along a geodesic $\eta$ on ${\mathcal D}^{s}_{\mu}(M)$
corresponds to the stability of a fluid flow $u=\dot{\eta}\circ\eta^{-1}$.
We recall that the definition of a {\it conjugate point}.

\begin{definition}
\label{definition of conjugate point}
(Conjugate point.)\ 
Let 
$D$ be a Riemannian manifold
and
$\eta(t):=\exp_{p}(tV)$ 
a geodesic
for some $V\in T_{p}D$,
where
$\exp_{p}:T_{p}D\to D$
is
the exponential map 
at $p\in D$.
Then 
we say that
$\eta(1)$ is
a {\it conjugate point} 
or {\it conjugate} to $p$
along $\eta$
if
the differential 
$T_{V}\exp_{p}:T_{V}(T_{p}D)\to T_{\eta(1)}D$ of the exponential map at $V$
is not bijective.
(In the case of  $\dim D=\infty$,
there are two reasons for a point to be conjugate to another.
See Remark \ref{monoepi} and Appendix 2.)
\end{definition}

We define the crucial value for the existence of conjugate points
by
\begin{equation}\label{first-M-criterion}
MC_{V,W}:=(\nabla_{V}[V,W]
+\nabla_{[V,W]}V,W)
\end{equation}
for 
$V,W\in T_{e}{\mathcal D}^{s}_{\mu}(M)$.
We call this value the \textit{Misio{\l}ek curvature} for $V$ and $W$.

\begin{remark}\label{reason for C^2}
If $[V,W]$ is an only of class $C^{0}$ ,
we cannot define 
$MC_{V,W}$ for $V,W\in T_{e}{\mathcal D}^{s}_{\mu}(M)$
since we have one more derivative
in $[V,W]$
by $\nabla_V$.
Therefore,
we require 
$V,W\in T_{e}{\mathcal D}^{s}_{\mu}(M)$
for 
$s>2+\frac{n}{2}$,
which implies
that
$V$ and $W$ are of class $C^{2}$
by
Sobolev embedding theorem.
\end{remark}

The importance of the Misio{\l}ek curvature is the following criterion for the existence of conjugate points,
which we call the M-criterion.

\begin{fact}[{\cite[Lemmas 2 and 3]{MconjT2}}]
\label{factM-criterion}
Let $M$ be a compact $n$-dimensional Riemannian manifold without boundary
and
$s>2+\frac{n}{2}$.
Suppose that 
$V\in T_{e}{\mathcal D}^{s}_{\mu}(M)$ is a time independent solution of the Euler equations \eqref{Eintro} on $M$
and
take a geodesic $\eta(t)$ on ${\mathcal D}^{s}_{\mu}(M)$ satisfying $V=  \dot{\eta}\circ\eta^{-1}$.
Then 
if
$W\in T_{e}{\mathcal D}^{s}_{\mu}(M)$
satisfies
$MC_{V,W}>0$,
there exists
a point conjugate
to $e\in {\mathcal D}^{s}_{\mu}(M)$ along $\eta(t)$ on $0\leq t\leq t^{*}$
for some $t^{*}>0$.
\end{fact}
\begin{remark}
This fact is not 
explicitly stated
but essentially proved in \cite[Lemmas 2 and 3]{MconjT2}.
See Appendix 1
for the proof of the case that $\dim M=2$.
In Appendix 1,
we clarify more 
the meaning of $W\in T_{e}{\mathcal D}^{s}_{\mu}(M)$ satisfying $MC_{V,W}>0$.
\end{remark}

We are ready to state our main theorems:
Let $M$ be a 2-dimensional ellipsoid or a sphere,
more precisely,
$M=M_{a}:=\{(x,y,z)\in{\mathbb R}^{3}
\mid
x^{2}+y^{2}=a^{2}(1-z^{2})
\}$ for some $a>1$ (having a bulge around its equatorial middle and is flattened at the poles) and $a=1$ (sphere).
We regard $M$ as a Riemannian manifold 
by the induced metric $g$ from ${\mathbb R}^{3}$.
We say that 
a vector field $V$ on $M$ 
is 
a
{\it zonal flow} 
if $V$ has the following form:
\begin{equation}
\label{zonal}
V
=F(z)(y\partial_{x}-x\partial_{y})
\end{equation} 
for some function $F:[-1,1]\to{\mathbb R}$.
In other words,
$V$ is a product of a function $F(z)$ and the flow 
of the rotation around the $z$-axis
(This flow is nothing more than a Killing vector field on $M_{a}$).
Recall that the support of a vector field of $V$ on $M$ 
is 
the closure of $\{x\in M\mid V(x)\neq 0\}$.
\begin{theorem}\label{main-theorem}
Suppose $s>3$ and  $a>1$.
For any
zonal flow
$V\in T_{e}{\mathcal D}^{s}_{\mu}(M_{a})$ whose support is contained in $M_{a}\backslash \{(0,0,1),(0,0,-1)\}$,
then there exists
$W\in T_{e}{\mathcal D}^{s}_{\mu}(M_{a})$
satisfying $MC_{V,W}>0$.
\end{theorem}
On the other hand, in the sphere case,
we have the following (cf. \cite{Lukatsky}):
\begin{theorem}\label{second-main-theorem}
Suppose $s>3$.
For any zonal flow $V\in T_{e}{\mathcal D}^{s}_{\mu}(S^2)$
and
any $W\in T_{e}{\mathcal D}^{s}_{\mu}(S^{2})$,
we have
$MC_{V,W}\leq 0$.
\end{theorem}
\begin{remark}
The M-criterion itself cannot be a necessary condition for ensuring the existence of a conjugate point.
If both $V$ and $W$ are Killing vector fields on a sphere, then this combination induces the existence of a conjugate point (see Remark 2 in Section 3 in \cite{MconjT2}). Thus  it would be important to clarify the relation between these 
Killing vector fields and the M-criterion.
\end{remark}
Since this study is interdisciplinary, we first try to explain differential geometry step by step, and then finally we prove the  main theorems.
Therefore,
we briefly recall basic facts
and prove some results
of the theory of diffeomorphism groups 
in the context of
inviscid fluid flows
in
Section \ref{preliminary}.
We discuss about our background manifolds,
which we call
rotationally symmetric manifolds,
in Section \ref{Sect RSM}
and 
apply the facts of Section \ref{preliminary} to our problem in Sections \ref{rotationally} and 
\ref{ellipsoid}.
Moreover,
we 
sophisticate 
the meaning of
$W\in T_{e}{\mathcal D}^{s}_{\mu}(M)$ satisfying $MC_{V,W}>0$
and prove the M-criterion in the case $\dim M=2$ in Appendix 1
and solve an apparent paradox
concerning the Fredholmness of the exponential map in the 2D case and
the M-criterion
in Appendix 2.

\section{Preliminary}
\label{preliminary}
In this section,
we recall 
the theory of
diffeomorhphism groups
in the context of 
inviscid fluid flows.
Our main references are
\cite{EMa} and \cite{Mstability}.
We also refer to \cite{Kambe}
for a well-organized review of this field.
Moreover,
the same theory is applied
in
\cite{Pearce}
for the
SQG equation.

Let $(M, g)$ be a compact $n$-dimensional Riemannian manifold without boundary
and ${\mathcal D}^{s}(M)$  
the group of Sobolev $H^{s}$ diffeomorphisms of $M$
and
${\mathcal D}_{\mu}^{s}(M)$
the subgroup of ${\mathcal D}^{s}(M)$
consisting volume preserving elements,
where $\mu$ is the volume form on $M$ defined by $g$.
If $s>1+\frac{n}{2}$,
the group ${\mathcal D}^{s}(M)$
can be 
given a structure of an infinite-dimensional weak Riemannian manifold
(see \cite{EMa})
and ${\mathcal D}^{s}_{\mu}(M)$ become its
weak Riemannian submanifold
(The term ``weak''
means that
the topology induced from the metric 
is weaker than
the original topology of
${\mathcal D}^{s}(M)$
or 
${\mathcal D}^{s}_{\mu}(M)$).
This weak Riemannian metric is given as follows:
The tangent space $T_{\eta}{\mathcal D}^{s}(M)$ 
of ${\mathcal D}^{s}(M)$ at a 
point $\eta\in{\mathcal D}^{s}(M)$ consists of all 
$H^{s}$ vector fields on $M$ which cover $\eta$, namely, 
all $H^{s}$ sections of the pullback bundle $\eta^{*}TM$.
Thus
for $x\in M$
and 
$V, W\in T_{\eta}{\mathcal D}^{s}(M)$,
we have 
$V(x),
W(x)
\in
T_{\eta(x)}M$.
Then we define an inner product on $T_{\eta}{\mathcal D}^{s}(M)$
by
\begin{eqnarray}
\left(V, W\right)
:=
\int_{M}
g(V(x),W(x))\mu(x)
\label{L2metric}
\end{eqnarray}
and set $|V|:=\sqrt{(V,V)}$.
Similarly,
$T_{\eta}{\mathcal D}^{s}_{\mu}(M)$ 
consists of all 
$H^{s}$ divergence-free vector fields on $M$ which cover 
$\eta\in{\mathcal D}_{\mu}^{s}(M)$.
Therefore, the metric \eqref{L2metric} induces a direct sum:
\begin{eqnarray}
T_{\eta}{\mathcal D}^{s}(M)=
T_{\eta}{\mathcal D}^{s}_{\mu}(M)
\oplus
\{
({\rm grad} f)
\circ 
\eta
\mid
f\in
H^{s+1}(M)\},
\label{directsum}
\end{eqnarray}
which follows from the fact that
the gradient
is 
the adjoint of 
the negative divergence.
Let
\begin{eqnarray*}
	P_{\eta}&:& T_{\eta}{\mathcal D}^{s}(M)
	\to T_{\eta}{\mathcal D}^{s}_{\mu}(M)
	\\
	Q_{\eta}&:& T_{\eta}{\mathcal D}^{s}(M)
	\to 
	\{
({\rm grad} f)
\circ 
\eta
\mid
f\in
H^{s+1}(M)\}
\end{eqnarray*}
be the projection to the first and second components of \eqref{directsum},
respectively.
Moreover,
we write $e\in{\mathcal D}^{s}(M)$ for the 
identity element of 
${\mathcal D}^{s}(M)$.
\begin{lem}
\label{PQ}
For 
$X,Y\in T_{e}{\mathcal D}^{s}(M)$,
we have
$$
\left(P_{e}X, P_{e}Y\right)
=\left(P_{e}X, Y\right)
=\left(X, P_{e}Y\right),
$$
$$
\left(Q_{e}X, Q_{e}Y\right)
=\left(Q_{e}X, Y\right)
=\left(X, Q_{e}Y\right).
$$
\end{lem}
\begin{proof}
This is clear
by the direct sum \eqref{directsum}.
\end{proof}
The metric \eqref{L2metric}
also
induces
the
right invariant Levi-Civita connections $\bar{\nabla}$ and
$\widetilde{\nabla}$ on ${\mathcal D}^{s}(M)$ and
${\mathcal D}^{s}_{\mu}(M)$,
respectively.
This is defined as follows:
Let 
$V, W$ be vector fields on ${\mathcal D}^{s}(M)$.
We write
$V_{\eta}\in T_{\eta}{\mathcal D}^{s}(M)$
for
the value of 
$V$
at $\eta\in {\mathcal D}^{s}(M)$.
Then we have $V_{\eta}\circ \eta^{-1}, W_{\eta}\circ \eta^{-1}\in T_{e}{\mathcal D}^{s}(M)$,
namely,
$V_{\eta}\circ \eta^{-1}$ and $W_{\eta}\circ \eta^{-1}$ are vector fields on $M$.
Moreover,
we have
$W_{\eta}\circ \eta^{-1}$ is a vector field of class $C^{1}$ on $M$
by Sobolev embedding theorem
and the assumption $s>1+\frac{n}{2}$.
Thus we can consider $\nabla_{V_{\eta}\circ \eta^{-1}}W_{\eta}\circ \eta^{-1}$,
where $\nabla$ is the Levi-Civita connection on $M$.
Take a path $\varphi$ on ${\mathcal D}^{s}(M)$ satisfying
$\varphi(0)=\eta$ and $V_{\eta}=\partial_t\varphi(0)\in T_{\eta}{\mathcal D}_{\mu}^{s}(M)$,
then we define
\begin{eqnarray}
(\bar\nabla_{V}W)_{\eta}
&:=&
\frac{d}{dt}\left(
W_{\varphi(t)}\circ\varphi^{-1}(t)
\right)|_{t=0}\circ\eta+(\nabla_{V_{\eta}\circ\eta^{-1}}
W_{\eta}\circ\eta^{-1})\circ\eta.
\label{nab1}
\end{eqnarray}
Moreover, if $V$ and $W$ are vector fields on 
${\mathcal D}^{s}_{\mu}(M)$,
we define
\begin{eqnarray}
(\widetilde{\nabla}_{V}W)_{\eta}
&:=&P_\eta(\bar{\nabla}_{V}W)_{\eta}
\label{nab2}.
\end{eqnarray}
These definitions are independent of the particular choice of $\varphi(t)$.
We note that
$(\bar\nabla_{V}W)_{\eta}=(\bar\nabla_{V}W)_{e}\circ \eta$
if $V$ and $W$ are right invariant
vector fields on ${\mathcal D}^{s}(M)$
(i.e., $\bar{\nabla}$ is right invariant).
This is because
if $W$ is right invariant,
or equivalently, if $W$ satisfies $W_{\eta}=W_{e}\circ \eta$ for any $\eta\in{\mathcal D}^{s}_{\mu}(M)$,
the first term of \eqref{nab1} vanishes.

Moreover, the right invariant Levi-Civita connection $\bar{\nabla}$  induces 
the curvature tensor $\bar{R}$ on ${\mathcal D}^{s}(M)$,
which is given by
\begin{eqnarray*}
\bar R_{\eta}(X,Y)Z
&=&
(\bar\nabla_X\bar\nabla_YZ)_{\eta}
-
(\bar\nabla_Y\bar\nabla_XZ)_{\eta}
-
(\bar\nabla_{[X,Y]}Z)_{\eta}
\end{eqnarray*}
for vector fields $X,Y$ and $Z$ on ${\mathcal D}^{s}(M)$.
As in the case of finite-dimensional Riemannian manifolds, this depends only on the values of $X,Y$ and $Z$ at $\eta$,
in other words, we can define $\bar R_{\eta}(X_{\eta},Y_{\eta})Z_{\eta}$ for $X_{\eta},Y_{\eta},Z_{\eta}\in T_{\eta}{\mathcal D}^{s}(M)$.
Therefore the right invariance of $\bar{\nabla}$ implies
\begin{eqnarray*}
\bar R_{\eta}(X_{\eta},Y_{\eta})Z_{\eta}
=
\left(
R(X_\eta\circ\eta^{-1},Y_\eta\circ\eta^{-1})
(Z_\eta\circ\eta^{-1})
\right)\circ\eta,
\end{eqnarray*}
where $R$ is the curvature of $M$.
Similarly, 
the right invariant Levi-Civita connection $\widetilde{\nabla}$  induces
the curvature tensor $\widetilde{R}$ on ${\mathcal D}^{s}_{\mu}(M)$,
which is given by
\begin{eqnarray*}
\widetilde{R}_\eta(X_\eta,Y_\eta)Z_\eta=(P_{e}\nabla_{X_e}P_{e}\nabla_{Y_e}Z_e-P_{e}\nabla_{Y_e}P_{e}\nabla_{X_e}Z_e-P_{e}\nabla_{[X_e,Y_e]}Z_e)\circ\eta,
\end{eqnarray*}
where $X_e=X_\eta\circ\eta^{-1}$.
These curvatures $\bar{R}$ and $\widetilde{R}$ are related by 
the Gauss-Codazzi equations:
\begin{equation}
(
\bar{R}
(X, Y
)Z,
W)
=
(
\widetilde{R}
(X, Y
)Z,
W)
+
(
Q{\nabla}_{X}Z, Q{\nabla}_{Y}W)
-
(
Q{\nabla}_{Y}Z,Q{\nabla}_{X}W
)
\label{Gauss-Codazzi}
\end{equation}
for any vector fields $X,Y,Z$ and $W$ on ${\mathcal D}^{s}_{\mu}(M)$.

A geodesic joining the identity element $e\in {\mathcal D}^{s}_{\mu}(M)$ and
$p\in{\mathcal D}^{s}_{\mu}(M)$ 
can be obtained 
from
a variational principle as a stationary point of the energy function:
\begin{eqnarray}
E(\eta)^{t_{0}}_{0}:=\frac{1}{2}
\int_{0}^{t_{0}}
\left|
\dot{\eta}(t)
\right|^{2}dt,
\label{Eint}
\end{eqnarray}
where $\eta$ is a curve on ${\mathcal D}^{s}_{\mu}(M)$
satisfying $\eta(0)=e$ and $\eta(t_{0})=p$
and we set $\dot{\eta}(t):=\partial_{t}\eta(t)\in T_{\eta(t)}{\mathcal D}^{s}_{\mu}(M)$.
Let $\xi(r,t):(-\varepsilon,\varepsilon)\times[0,t_{0}]\to {\mathcal D}^{s}_{\mu}(M)$
be
a variation of a geodesic $\eta(t)$ with fixed end points,
namely,
it satisfies
$\xi(r,0) = \eta(0)$, $\xi(r,t_{0}) = \eta(t_{0})$ and $\xi(0,t)=\eta(t)$ for $t\in[0,t_{0}]$.
We sometimes write $\xi_{r}(t)$ for $\xi(r,t)$.
Let
$X(t):=\partial_{r}\xi(r,t)|_{r=0}\in T_{\eta(t)}{\mathcal D}^{s}_{\mu}(M)$ be the associated vector field
on ${\mathcal D}^{s}_{\mu}(M)$. 
Then 
the first and the second variations of the above integral are given by
\begin{alignat}{1}
0=
E'(\eta)_{0}^{t_{0}}(X)
=&
(X(t_{0}),\dot{\eta}(t_{0}))
-
(X(0),\dot{\eta}(0))\nonumber\\
&-
\int_{0}^{t_{0}}
(X(t),
\widetilde{\nabla}
_{\dot{\eta}(t)}\dot{\eta}(t)
)
dt,
\nonumber\\
E''(\eta)^{t_0}_{0}(X,X)
=&
\int_{0}^{t_{0}}
\{
(
\widetilde{\nabla}_{\dot{\eta}} X,
\widetilde{\nabla}_{\dot{\eta}} X)
-
(
\widetilde{R}_{\eta}
(X, \dot{\eta}
)\dot{\eta},
X)
\}dt.
\label{E''}
\end{alignat}

The reason why the geometry of ${\mathcal D}_{\mu}^{s}(M)$ 
is important
is 
that
geodesics in ${\mathcal D}^{s}_{\mu}(M)$ 
correspond to inviscid fluid flows on $M$,
which was first remarked by V. I. Arnol'd \cite{A}.
This correspondence is accomplished in the following way:
If $\eta(t)$ is a geodesic on ${\mathcal D}^{s}_{\mu}(M)$ (i.e., $\widetilde{\nabla}_{\dot{\eta}}\eta=0$)
joining $e$ and $\eta(t_{0})$,
a time dependent vector field on $M$ defined by
$u(t):=\dot{\eta}(t)\circ \eta^{-1}(t)$ 
is a solution to the Euler equations on $M$:
\begin{align} \nonumber 
&\partial_tu + \nabla_{u} u = - \operatorname{grad}p
\qquad 
t \in [0,t_{0}], 
\\ 
\label{Epre} 
&\operatorname{div} u=0,
\\ 
\nonumber 
&u|_{t=0} = \dot{\eta}(0), 
\end{align} 
with a scalar  function (pressure) $p(t)$ determined by $u(t)$.
Here
$\operatorname{grad}p$ (resp. $\operatorname{div}u$) is the gradient (resp.\:divergent) of $p$ (resp.\:$u$) with respect to the Riemannian metric $g$ of $M$.
In this context,
the existence of conjugate points along a geodesic $\eta$ (see Definition \ref{definition of conjugate point})
corresponds to the stability (in a short time) of a fluid flow $u=\dot{\eta}\circ\eta^{-1}$.
\begin{remark}
\label{monoepi}
For an infinite-dimensional Riemannian manifold $D$,
there are two reasons to be a conjugate point 
\cite{Gross}.
Let 
$\widetilde{\exp}_{\eta(0)}:T_{\eta(0)}D\to D$ be the exponential map of $D$ 
and
$\eta(t):=\widetilde{\exp}_{\eta(0)}\: tV$ a geodesic for some $V\in T_{\eta(0)}D$.
Then,
we say that
$\eta(1)$ is {\it monoconjugate} (resp.\:{\it epiconjugate})
if
the differential $T_{V} \widetilde{\exp}_{\eta(0)}$ of the exponential map at $V$
is not injective (resp.\:not surjective).
Of course,
monoconjugate points are important 
from the view point of the stability of a fluid flow.
However, 
the following fact implies that 
monoconjugate points and epiconjugate points along any geodesic on ${\mathcal D}^{s}_{\mu}(M)$ coincide in the 2D case.
\begin{fact}[{\cite[Theorem 1]{EMis}}]
\label{Fredholm}
Let $M$ be a compact 2-dimensional Riemannian manifold without boundary.
Then,
the exponential map
$\widetilde{\exp}_{e}: T_{e}{\mathcal D}^{s}_{\mu}(M) \to{\mathcal D}^{s}_{\mu}(M)$,
which is induced by 
the Levi-Civita connection $\widetilde{\nabla}$,
is a nonlinear Fredholm map.
More precisely,
for any $V\in T_{e}{\mathcal D}^{s}_{\mu}$,
the derivative $T_{V}\widetilde{\exp}_{e}:T_{V}(T_{e}{\mathcal D}^{s}_{\mu})\simeq T_{e}{\mathcal D}^{s}_{\mu} \to T_{\widetilde{\exp}_{e}(V)}{\mathcal D}^{s}_{\mu}$
is a bounded Fredholm operator of index zero.
\end{fact}
\end{remark}
\begin{remark}
For example,
see \cite{BMP,LL,Presten}
for further studies of singularities of the exponential map.
\end{remark}

In order to consider the existence of a conjugate point,
we start with the following proposition,
which is proved by Misio{\l}ek \cite[Lemma 2]{MconjT2}
in the case of $M={\mathbb T}^{2}$.
Although Misio{\l}ek's proof can be applied to the case that $M$ is arbitrary
compact $n$-dimensional manifold  
without boundary,
we prove the proposition in such case
for the sake of completeness.

\begin{prop}
\label{formulanegative}
Let $M$ be a compact n-dimensional Riemannian manifold without boundary
and
$V,W\in T_{e}{\mathcal D}^{s}_{\mu}(M)$.
Suppose that 
$s>2+\frac{n}{2}$ and that
$V$ is a time independent solution of the Euler equations \eqref{Epre} on $M$.
Take a geodesic $\eta(t)$ on ${\mathcal D}^{s}_{\mu}(M)$ satisfying $V=  \dot{\eta}\circ\eta^{-1}$ as a vector field on $M$
and
a smooth function $f:[0,t_{0}]\to {\mathbb R}$
satisfying $f(0)=f(t_{0})=0$
for some $t_{0}>0$.
Then,
we have
$$
E''(\eta)^{t_{0}}_{0}(\widetilde W,\widetilde W)
=
\int_{0}^{t_{0}}
\left(
\dot{f}^{2}
|W|^{2}
-
f^{2}
(
\nabla_{V}[V,W]
+\nabla_{[V,W]}V,W)
\right)dt,
$$
where $|W|^2:=(W,W)_{T_e\mathcal {D}^s_\mu(M)}$
and
$\widetilde{W}$ is a vector field on ${\mathcal D}^{s}_{\mu}(M)$ along $\eta$
defined by
$\widetilde{W}_{\eta(t)}:=f(t)(W\circ \eta(t))\in T_{\eta(t)}{\mathcal D}^{s}_{\mu}(M)$.
\end{prop}
For the proof of this proposition,
we need the following three lemmas.
\begin{lem}
\label{skew}
Let $X, Y\in T_{e}{\mathcal D}^{s}(M)$
and
$W\in T_{e}{\mathcal D}^{s}_{\mu}(M)$.
Then,
we have
\begin{eqnarray*}
(\nabla_{W}X,Y)
=
-
(X,\nabla_{W}Y).
\end{eqnarray*}
\end{lem}
We omit the proof of this lemma, because this is easy.
\begin{lem}
\label{VWWV}
Let $V, W\in T_{e}{\mathcal D}^{s}_{\mu}(M)$ 
and
$X\in T_{e}{\mathcal D}^{s}(M)$.
Then,
we have
\begin{eqnarray*}
(
\nabla_{V}W,
Q_{e}X
)
_{T_{e}{\mathcal D}^{s}(M)}
=
(
\nabla_{W}V,
Q_{e}X
)
_{T_{e}{\mathcal D}^{s}(M)}.
\end{eqnarray*}
\end{lem}
\begin{proof}
This is an easy consequence of $\nabla_{V}W-\nabla_{W}V=[V,W]\in T_{e}{\mathcal D}^{s}_{\mu}(M)$.
\end{proof}
\begin{lem}
\label{volpre}
For any $V,W\in T_{e}{\mathcal D}^{s}_{\mu}(M)$
and
$\eta\in{\mathcal D}^{s}_{\mu}(M)$,
we have
\begin{eqnarray*}
(V,W)
=
(V\circ \eta,
W\circ \eta)
\end{eqnarray*}
\end{lem}
\begin{proof}
This follows from the definition of the metric on ${\mathcal D}^{s}_{\mu}(M)$
and $\eta \in {\mathcal D}^{s}_{\mu}(M)$.
\end{proof}

\begin{proof}[Proof of Proposition \ref{formulanegative}]
We follow the same strategy  in \cite[Lemma 2]{MconjT2}.

The second variation $E''$ along ${\widetilde W}$ can be expressed as
\begin{eqnarray}
E''(\eta)^{t_{0}}_{0}(\widetilde W,\widetilde W)
=
\int_{0}^{t_{0}}
\{
(
\widetilde{\nabla}_{\dot{\eta}} \widetilde{W},
\widetilde{\nabla}_{\dot{\eta}} \widetilde{W})
-
(
\widetilde{R}_{\eta}
(\widetilde{W}, \dot{\eta}
)\dot{\eta},
\widetilde{W})
\}dt.
\label{second}
\end{eqnarray}
For the first term,
we have
\begin{eqnarray*}
\widetilde{\nabla}_{\dot{\eta}} \widetilde{W}
&=&
P_{\eta}\bar{\nabla}_{\dot{\eta}}\widetilde{W}
=
P_{\eta}
\left(
\frac{d}{dt}\left(
\widetilde{W}_{\eta}\circ\eta^{-1}
\right)\circ\eta+(\nabla_{\dot{\eta}\circ\eta^{-1}}
\widetilde{W}\circ\eta^{-1})\circ\eta
\right)\\
&=&
P_{\eta}
\left(
\frac{d}{dt}
\left(
f
W
\right)\circ\eta+(\nabla_{V}
\left(
fW
\right)
)\circ\eta
\right)\\
&=&
P_{\eta}
\left(
\dot{f}
\cdot
(W\circ\eta)
+
(
f\nabla_{V}W
)\circ\eta
\right)
\end{eqnarray*}
by \eqref{nab1}, \eqref{nab2}.
We note that 
$\nabla_{V}
(fW)=f\nabla_{V}W$
follows from the fact that
$f$ depends only on the time variable $t\in[0,t_{0}]$.
Moreover,
we have
\begin{eqnarray*}
\widetilde{\nabla}_{\dot{\eta}} \widetilde{W}
&=&
(
\dot{f}
\cdot
W
+
f
\cdot
P_{e}
\nabla_{V}
W)\circ\eta
\end{eqnarray*}
by $P_{\eta}(W\circ\eta)=(P_{e}W)\circ \eta=W\circ \eta$.
Thus,
Lemma \ref{volpre}
implies
\begin{eqnarray*}
(
\widetilde{\nabla}_{\dot{\eta}} \widetilde{W},
\widetilde{\nabla}_{\dot{\eta}} \widetilde{W})
&=&
\dot{f}^{2}|W|^{2}
+
2f\dot{f}(W,P_{e}\nabla_{V}W)
+
f^{2}
|P_{e}\nabla_{V}W|^{2},
\end{eqnarray*}
where $|W|^2:=(W,W)$.
The direct sum \eqref{directsum}
and
$\operatorname{div} W=0$
impliy
$$(W,P_{e}\nabla_{V}W)=(W,(P_{e}+Q_{e})\nabla_{V}W)=(W,\nabla_{V}W),$$
which vanishes because
$(W,\nabla_{V}W)=
-(W,\nabla_{V}W)$
by
Lemma \ref{skew}.
Thus, we have
\begin{eqnarray}
(
\widetilde{\nabla}_{\dot{\eta}} \widetilde{W},
\widetilde{\nabla}_{\dot{\eta}} \widetilde{W})
&=&
\dot{f}^{2}|W|^{2}
+
f^{2}
|P_{e}\nabla_{V}W|^{2}\label{firstterm}\\
&=&
\dot{f}^{2}|W|^{2}
+
f^{2}
(\nabla_{V}W,P_{e}\nabla_{V}W)
\nonumber
\end{eqnarray}
by Lemma \ref{PQ}.
For the second term of \eqref{second},
we have
\begin{eqnarray}
(
\widetilde{R}_{\eta}
(\widetilde{W}, \dot{\eta}
)\dot{\eta},
\widetilde{W}
)
&=&
(
\widetilde{R}_{\eta}
(f\cdot(W\circ \eta),
(V\circ\eta)
)
(V\circ\eta),
f\cdot(W\circ \eta)
)\nonumber\\
&=&
f^{2}
(
\widetilde{R}_{e}
(W, V
)V,
W).
\label{secondterm}
\end{eqnarray}
by
the right invariance of $\widetilde{R}$.
The Gauss-Codazzi equations
\eqref{Gauss-Codazzi} 
imply 
\begin{eqnarray*}
(
\bar{R}_{e}
(W, V
)V,
W)
&=&
(
\widetilde{R}_{e}
(W, V
)V,
W)
+
(
Q_{e}{\nabla}_{W}V, Q_{e}{\nabla}_{V}W)
-
(
Q_{e}{\nabla}_{V}V,Q_{e}{\nabla}_{W}W
).
\end{eqnarray*}
Therefore, by Lemmas \ref{PQ}, \ref{skew} and \ref{VWWV},
we have
\begin{eqnarray*}
(
\widetilde{R}_{e}
(W, V
)V,
W)
&=&
(
\bar{R}_{e}
(W, V
)V,
W)
-
(
\nabla_{W}V,
Q_{e}\nabla_{V}W)
+
(
Q_{e}\nabla_{V}V,
\nabla_{W}W
)\\
&=&
(\nabla_{W}\nabla_{V}V-\nabla_{V}\nabla_{W}V
-\nabla_{[W,V]}V,W)\\
&&-
(
\nabla_{V}W, 
Q_{e}\nabla_{V}W)
+
(
Q_{e}\nabla_{V}V,\nabla_{W}W
).
\end{eqnarray*}
We note that
$V$ is a time independent solution of \eqref{Epre},
namely, $Q_{e}\nabla_{V}V=\nabla_{V}V$.
Thus,
Lemma \ref{skew}
and
$\nabla_{V}W-\nabla_{W}V=[V,W]$ imply
\begin{eqnarray}
\label{R}
&&-(
\widetilde{R}_{e}
(W, V
)V,
W)
+
(\nabla_{V}W,
P_{e}\nabla_{V}W)\\
&=&
-
(\nabla_{W}\nabla_{V}V-\nabla_{V}\nabla_{W}V
-\nabla_{[W,V]}V,W)
+
(
\nabla_{V}W, \nabla_{V}W)
-
(
\nabla_{V}V,\nabla_{W}W
)
\nonumber\\
&=&
-
(
\nabla_{V}[V,W]
+\nabla_{[V,W]}V,W).
\nonumber
\end{eqnarray}
Therefore,
by \eqref{firstterm}, \eqref{secondterm} and \eqref{R},
we have
\begin{eqnarray}
E''(\eta)^{t_{0}}_{0}(\widetilde W,\widetilde W)
&=&
\int_{0}^{t_{0}}
\left(
\dot{f}^{2}|W|^{2}
+
f^{2}
\left(
(\nabla_{V}W,P_{e}\nabla_{V}W)
-(
\widetilde{R}_{e}
(W, V
)V,
W)
\right)
\right)\nonumber\\
&=&
\int_{0}^{t_{0}}
\left(
\dot{f}^{2}|W|^{2}
-
f^{2}
(
\nabla_{V}[V,W]
+\nabla_{[V,W]}V,W)
\right)dt.
\label{eq-nabla-to-nabla}
\end{eqnarray}
This completes the proof.
\end{proof}

From the above lemma, we can naturally extract the key value $MC_{V,W}$: 
\begin{eqnarray}
\label{eq-def-M-criterion}
MC_{V,W}
&:=&
(\nabla_{V}[V,W]
+\nabla_{[V,W]}V,W)
\\
&=&(\widetilde{R}_{e}(W,V)V,W)
-
|\widetilde{\nabla}_{V}W|_{T_{e}{\mathcal D}^{s}_{\mu}}^{2}
\nonumber
\end{eqnarray}
for $W\in T_{e}{\mathcal D}^{s}_{\mu}(M)$
and 
a time independent solution $V\in T_{e}{\mathcal D}^{s}_{\mu}(M)$ of 
the Euler equations \eqref{Epre} on $M$.
The second equality follows from 
\eqref{nab2}
and the calculation in \eqref{eq-nabla-to-nabla}.
We call $MC_{V,W}$ the ``Misio{\l}ek curvature".
This value is the crucial in this paper, since $MC_{V,W}>0$ ensures 
the existence of a conjugate point (see Fact \ref{factM-criterion} and Corollary \ref{conjugate}).
\begin{remark}
By 
\eqref{eq-def-M-criterion},
it is obvious that
$MC_{V,W}>0$ implies 
the sectional curvature 
$(\widetilde{R}_{e}(W,V)V,W)$
is positive.
Moreover,
we have
$MC_{V,W+cV}=MC_{V,W}$ for any $c\in{\mathbb R}$.
Thus $MC_{V,*}:T_{e}{\mathcal D}^{s}_{\mu}(M)\to {\mathbb R}$
should be defined on 
$V^{\perp}
:=
\{W\in T_{e}{\mathcal D}^{s}_{\mu}(M)\mid (V,W)=0\}$.
\end{remark}
\begin{cor}
\label{negative}
Let $M$ be a compact $n$-dimensional Riemannian manifold without boundary
and
$s>2+\frac{n}{2}$.
Suppose that 
$V\in T_{e}{\mathcal D}^{s}_{\mu}(M)$ is a time independent solution of the Euler equations  \eqref{Epre} on $M$
and that
$W\in T_{e}{\mathcal D}^{s}_{\mu}(M)$
satisfies
$MC_{V,W}>0$.
Take a geodesic $\eta(t)$ on ${\mathcal D}^{s}_{\mu}(M)$ satisfying $V=  \dot{\eta}\circ\eta^{-1}$ as a vector field on $M$
and $k\in{\mathbb R}_{>0}$.
Define a positive number 
$t_{V,W,k}>0$ and
a smooth function $f_{V,W,k}:[0,t_{V,W,k}]\to {\mathbb R}$
satisfying $f_{V,W,k}(0)=f_{V,W,k}(t_{V,W,k})=0$
by
\begin{eqnarray*}
t_{V,W,k}:=\pi|W|\sqrt{\frac{k}{MC_{V,W}}},\quad
f_{V,W,k}(t):=\sin \left(
\frac{t}{|W|}
\sqrt{\frac{MC_{V,W}}{k}}\right).
\end{eqnarray*}
Then we have
$$
E''(\eta)^{t_{V,W,k}}_{0}({\widetilde W}^{k},{\widetilde W}^{k})
=
\frac{\pi}{2}
(1-k)
\sqrt{\frac{MC_{V,W}}{k}},
$$
where
${\widetilde W}^{k}$ is a vector field on ${\mathcal D}^{s}_{\mu}(M)$ along $\eta$
defined by
$$
{\widetilde W}^{k}_{\eta(t)}:=
f_{V,W,k}(t)
(W\circ \eta(t))\in
T_{\eta(t)}{\mathcal D}^{s}_{\mu}(M).
$$
In particular,
if $k>1$
we have
$
E''(\eta)^{t_{V,W,k}}_{0}({\widetilde W}^{k},{\widetilde W}^{k})<0$
and
if $k=1$
we have
$
E''(\eta)^{t_{V,W,k}}_{0}({\widetilde W}^{k},{\widetilde W}^{k})=0$.
\end{cor}
\begin{proof}
Proposition \ref{formulanegative}
implies
\begin{eqnarray*}
&&E''(\eta)^{t_{V,W,k}}_{0}({\widetilde W}^{k},{\widetilde W}^{k})\\
&=&
\int_{0}^{t_{V,W,k}}
\left(
\dot{f}_{V,W,k}^{2}
|W|^{2}
-
MC_{V,W}
f_{V,W,k}^{2}
\right)dt\\
&=&
MC_{V,W}
\int_{0}^{\pi|W|\sqrt{\frac{k}{MC_{V,W}}}}
\left(
\frac{1}{k}
\cos^{2}
\left(
\frac{t}{|W|}
\sqrt{\frac{MC_{V,W}}{k}}\right)
-
\sin^{2}
\left(
\frac{t}{|W|}
\sqrt{\frac{MC_{V,W}}{k}}\right)
\right)
dt\\
&=&
MC_{V,W}
\int_{0}^{\pi}
\left(
\frac{1}{k}
\cos^{2}
x
-
\sin^{2}
x
\right)
|W|
\sqrt{\frac{k}{MC_{V,W}}}
dx\\
&=&
\frac{\pi}{2}
|W|
(1-k)
\sqrt{\frac{MC_{V,W}}{k}}.
\end{eqnarray*}
This completes the proof.
\end{proof}

\section{Rotationally symmetric manifolds}
\label{Sect RSM}
In this section,
we define
the notion of 
{\it rotationally symmetric manifolds},
which we take as our ``back ground manifolds''
in the later sections.
Our background manifold
is a sphere or an ellipsoid
in the main application.
We refer to \cite[Section 1.3]{Peter}
for the contents of this section.

Let
$I_{d}:=(-d,d)\subset {\mathbb R}$
be an open interval 
for some $d\in {\mathbb R}$
and
$$
c(r):=(c_{1}(r),0,c_{2}(r))
:
[-d,d]\to {\mathbb R}^{3}
$$
a smooth curve.
Suppose that
$c(r)$ 
satisfies
\begin{equation}
\begin{array}{cl}
{\rm (1)}&
c_{1}(r)>0\text{ for all }r\in I_{d},\\
{\rm (2)}&
c_{1}(d)=c_{1}(-d)=0,\\
{\rm (3)}&
\dot{c}_{2}(r):=\frac{dc_{2}}{dr}(r)>0
\text{ for all }r\in I_{d},\\
{\rm (4)}&
r
\text{
is a length parameter (i.e.,
}
\dot{c}_{1}^{2}+\dot{c}_{2}^{2}=1
\text{)}.\\
\end{array}
\label{R'(c).cond.}
\end{equation}
The condition (3) of \eqref{R'(c).cond.} means that
$c_{2}(r)$ is a monotonically increasing function.
Rotating this curve $c(r)$
with respect to the $z$-axis,
we obtain a surface of revolution:
\begin{equation}
R'(c):=\{
(c_{1}(r)\cos\theta,c_{1}(r)\sin\theta,c_{2}(r))\mid
r\in I_{d},\:
\theta\in {\mathbb R}
\}
\qquad
\subset
{\mathbb R}^{3}.
\label{def R'}
\end{equation}
We want to obtain a sufficient (and in fact necessary) condition so that
the closure $R(c):=cl(R'(c))$ 
has a smooth Riemannian manifold structure 
induced from the usual Riemannian structure of ${\mathbb R}^{3}$.

\begin{lem}
\label{R(c)}
Suppose that
\begin{equation}
\frac{dc_{1}}{dr}(d)=
\frac{dc_{1}}{dr}(-d)
=1,
\text{ and
}
\frac{d^{2n}c_{1}}{dr^{2n}}(d)=
\frac{d^{2n}c_{1}}{dr^{2n}}(-d)=0
\text{ for any }
n\in{\mathbb Z}_{>0}
\label{cond.der.}
\end{equation}
Then
$R(c):=cl(R'(c))$
has a smooth Riemannian manifold structure 
with 
the induced metric from ${\mathbb R}^{3}$.
\end{lem}
\begin{proof}
By the definition of $R(c)$,
it is clear that
$$
R(c)\backslash R'(c)=
\{
c(d),
c(-d)
\}.
$$
We only prove that $c(-d)$
is not singular point of $R(c)$.
The case for
$c(d)$ can be proved in the similar way.

We first calculate the Riemannian metric $g_{R'}$
on $R'(c)$ induced 
from the usual Riemannian metric of ${\mathbb R}^{3}$
in the coordinate system
$(r,\theta)$.
Define
\begin{equation*}
\begin{array}{cccccc}
\phi
&
:
&
I_{d}\times {\mathbb R}
&
\to
&
R'(c)
&\subset R(c)=M
\\
&&
(r,\theta)
&
\mapsto
&
(c_{1}(r)\cos\theta,
c_{1}(r)\sin\theta,
c_{2}(r)
)
&
\end{array}
\end{equation*}
Then,
we have
\begin{alignat*}{1}
\phi_{*}(\partial_{r})=&
\dot{c}_{1}(r)\cos\theta
\partial_{x}
+
\dot{c}_{1}(r)\sin\theta
\partial_{y}
+
\dot{c}_{2}(r)
\partial_{z},
\\
\phi_{*}(\partial_{\theta})=&
-c_{1}(r)\sin\theta
\partial_{x}
+
c_{1}(r)\cos\theta
\partial_{y},
\end{alignat*}
where $\phi_{*}$ denotes the push out.
Then,
it follows from an easy calculation 
that
\begin{equation}
g_{R'}(\partial_{r},\partial_{r})=
\dot{c}_{1}^{2}
+
\dot{c}_{2}^{2}
=1,
\quad
g_{R'}(\partial_{r},\partial_{\theta})=0,
\quad
g_{R'}(\partial_{\theta},\partial_{\theta})=c_{1}^{2}.
\label{gR' in rtheta}
\end{equation}
Next,
we introduce a coordinate system
$$(a,b):=(t\cos\theta,t\sin\theta)\quad
\text{where
}
t:=r+d.
$$
Note that $t^{2}=a^{2}+b^{2}$
and $r\to -d$ corresponds to
$t\to 0$.
Then,
we have
\begin{equation}
\partial_{r}
=
\cos\theta\partial_{a}
+
\sin\theta\partial_{b},
\quad
\partial_{\theta}
=
-t\sin\theta\partial_{a}
+
t\cos\theta\partial_{b},
\label{partial ab}
\end{equation}
or equivalently,
\begin{equation}
\partial_{a}=
\cos\theta
\partial_{r}
-
\frac{\sin\theta}
{t}\partial_{\theta},
\qquad
\partial_{b}=
\sin\theta
\partial_{r}
+
\frac{\cos\theta}
{t}\partial_{\theta}.
\label{partial relation}
\end{equation}
Combining \eqref{gR' in rtheta}
and \eqref{partial relation},
we have
\begin{alignat}{1}
g_{R'}(\partial_{a},\partial_{a})=&
1
+
\left(
\frac{c_{1}^{2}}
{t^{2}}
-
1
\right)
\sin^{2}\theta
=
1
+
\frac{c_{1}^{2}-t^{2}}
{t^{4}}
b^{2}
,
\nonumber
\\
g_{R'}(\partial_{a},\partial_{b})=&
\left(
1
-
\frac{c_{1}^{2}}
{t^{2}}
\right)
\sin\theta\cos\theta
=
\frac{c_{1}^{2}-t^{2}}
{t^{4}}
ab
,
\label{gR' ab}
\\
g_{R'}(\partial_{b},\partial_{b})=&
1
+
\left(
\frac{c_{1}^{2}}
{t^{2}}
-
1
\right)
\cos^{2}\theta
=
1
+
\frac{c_{1}^{2}-t^{2}}
{t^{4}}
a^{2}.
\nonumber
\end{alignat}
Then,
considering a Taylor expansion of $c_{1}$,
we obtain that
all functions of \eqref{gR' ab}
is smooth at $(a,b)=0$
if $c_{1}$ satisfies (2) of \eqref{R'(c).cond.}
and
\eqref{cond.der.}.
This completes the proof.
\end{proof}
\begin{definition}
\label{Def RSM}
Let $M$ be a $2$-dimensional Riemannian submanifold of ${\mathbb R}^{3}$.
We say $M$ is a
{\it rotationally symmetric manifold}
if
$M$ is isometric to $R(c)$
(see \eqref{def R'} for the definition)
for some 
smooth curve
$c(r): [-d,d]\to {\mathbb R}^{3}$
satisfying \eqref{R'(c).cond.}
and
\eqref{cond.der.}.
\end{definition}

\section{Computations on Rotationally symmetric manifolds}
\label{rotationally}
In this section,
we apply the results 
in Section \ref{preliminary} 
to the case that 
$M$ is a
compact 2-dimensional
rotationally symmetric manifold,
which is defined in Section \ref{Sect RSM}.
Our main background manifold
is a sphere or an ellipsoid.

Let 
$M$ be a 
rotationally symmetric manifold with a Riemannian metric $g_{M}$. 
See Definition \ref{Def RSM}.
We use the same notations in Section \ref{Sect RSM}.
In particular,
$I_{d}:=(-d,d)\subset {\mathbb R}$ is an open interval,
where $d\in {\mathbb R}_{>0}$
and
\begin{equation*}
\begin{array}{cccccc}
\phi
&
:
&
I_{d}\times I_{\pi}
&
\DistTo
&
R'(c)
&\subset R(c)=M
\\
&&
(r,\theta)
&
\mapsto
&
(c_{1}(r)\cos\theta,
c_{1}(r)\sin\theta,
c_{2}(r)
)
&
\end{array}
\end{equation*}
is a local coordinate of $M$.
Note that
$c(r)$ satisfies 
$\dot{c}_{1}(r)^{2}+
\dot{c}_{2}(r)^{2}=1$
for any $r\in I_{d}$,
namely,
$c(r)$ is parameterized by arc length.
Then,
we obtain (see \eqref{gR' in rtheta})
$$
g_{M}(\partial_{r},\partial_{r})=
1,
\quad
g_{M}(\partial_{r},\partial_{\theta})=0,
\quad
g_{M}(\partial_{\theta},\partial_{\theta})=c_{1}^{2}
$$
and
$$
\mu=c_{1}(r)dr\wedge d\theta.
$$
This implies
\begin{equation}
(V,W)
=
\int_{-d}^{d}
\int_{-\pi}^{\pi}
\left(
V_{1}W_{1}
+
V_{2}W_{2}
c_{1}^{2}
\right)c_{1}
d\theta dr
\label{integration}
\end{equation}
for
$V=
V_{1}\partial_{r}
+
V_{2}\partial_{\theta}$
and
$W=
W_{1}\partial_{r}
+
W_{2}\partial_{\theta}$,
which are elements of $T_{e}{\mathcal D}^{s}_{\mu}(M)$.

For a time dependent vector field 
$u$ and a time dependent scalar valued function $p$, 
the Euler equations of an incompressible and inviscid fluid on  $M$ 
are as follows: 
\begin{align} \nonumber 
&\partial_tu + \nabla_{u} u = - \operatorname{grad}p
\qquad 
t \geq 0, \; 
\\ 
\label{Erot} 
&\operatorname{div} u=0,
\\ 
\nonumber 
&u|_{t=0} = u_0, 
\end{align} 
where $\operatorname{grad}p$ (resp.\:$\operatorname{div}u$) is the gradient (resp.\:divergent) of $p$ (resp.\:$u$) with respect to $g_{M}$
and $\nabla$ is the Levi-Civita connection of $g_{M}$.
In the local coordinate system $(r,\theta)$,
these are given by
\begin{eqnarray*}
\operatorname{grad}p&=&\partial_{r}p\partial_{r}+c_{1}^{-2}\partial_{\theta}p\partial_{\theta},\\
\operatorname{div}u&=&(\partial_{r}+c_{1}^{-1}\partial_{r}c_{1})u_{1}
                    +\partial_{\theta}u_{2}
\end{eqnarray*}         
for $u=u_{1}\partial_{r}+u_{2}\partial_{\theta}$.

Recall that
we call
a vector field $V$ on $M$ 
a
{\it zonal flow} 
if $V$ has the following form:
\begin{equation}
\label{sec4:zonal}
V
=F(r)\partial_{\theta}
\end{equation} 
for some function $F:I_{d}\to{\mathbb R}$.
See also \eqref{zonal}.
Take a geodesic $\eta(t)$ of ${\mathcal D}^{s}_{\mu}(M)$ such that
\begin{equation*}
  \dot{\eta}(t)\circ\eta^{-1}(t)
=V
\end{equation*}
as a vector field on $M$.
Because
$V$ is a time independent solution of \eqref{Erot},
we have $\eta(t)=\widetilde{\exp}_{e}(tV)$.
We now compute the Misio{\l}ek curvature,
namely,
$MC_{V,W}:=(\nabla_{V}[V,W]
+\nabla_{[V,W]}V,W)$.
\begin{prop}
\label{formula}
Let $s>3$
and
$V\in T_{e}{\mathcal D}^{s}_{\mu}(M)$ a zonal flow.
For 
$W\in T_{e}{\mathcal D}^{s}_{\mu}(M)$,
we have
\begin{eqnarray*}
&&
MC_{V,W}
\\
&=&
\int_{-d}^{d}
\int_{-\pi}^{\pi}
F^{2}c_{1}
\bigg(-
\left(
\partial_{\theta}W_{1}
\right)^{2}
-
c_{1}^{2}
\left(
\partial_{r}W_{1}
\right)^{2}
+
\Big(
\left(\partial_{r}c_{1}
\right)^{2}
-
c_{1}\partial_{r}^{2}c_{1}
\Big)
W_{1}^{2}
\bigg)
d\theta dr,
\end{eqnarray*}
where
$V=F(r)\partial_{\theta}$
and
$W=W_{1}\partial_{r}+W_{2}\partial_{\theta}$.
\end{prop}

\begin{proof}[Proof of Proposition \ref{formula}]
Recall that
the suffix $1$ is corresponding to $r$ and $2$ is corresponding to $\theta$.
Let $\Gamma^{k}_{ij}\:(1\leq i,j,k\leq 2)$ be the Christoffel symbols,
which is given by
\begin{equation}
\Gamma^k_{ij}=\frac{1}{2}\left(g^{k1}(\partial_ig_{j1}+\partial_jg_{i1}-\partial_1g_{ij})
+g^{k2}(\partial_ig_{j2}+\partial_jg_{i2}-\partial_2g_{ij})\right).
\label{Christoffel}
\end{equation}
Here we write $g^{-1}=(g^{ij})$ for the inverse of $g$.
In our setting, we have 
\begin{equation}
\Gamma^{1}_{22}=-c_{1}\partial_r c_{1}, \quad
\Gamma^{2}_{12}=\frac{\partial_r c_{1}}{c_{1}},\quad 
\Gamma^{2}_{21}=\frac{\partial_r c_{1}}{c_{1}}.
\label{Christoffel}
\end{equation}
The other symbols are zero.
Then, by the definition, we have
\begin{equation*}
\nabla_v w=\sum_k
\left\{vw_{k}+\sum_{ij}\Gamma^k_{ij}v_{i}w_{j}\right\}\partial_k
\end{equation*}
for $v=\sum_i v_{i}\partial_i$ and $w=\sum_{j}w_{j}\partial_j$.

By direct calculation, 
we have 
\begin{eqnarray*}
[V,W]
&=&
\left[F\partial_{\theta}W_{1}\right]\partial_{r}
+
\left[
F\partial_\theta W_{2}
-
W_{1}\partial_{r}F
\right]\partial_{\theta}.
\end{eqnarray*}
Also,
we have
\begin{eqnarray*}
\nabla_{[V,W]}V
&=&
\left[
\Gamma^{1}_{22}
(
F\partial_\theta W_{2}
-
W_{1}\partial_{r}F
)
F
\right]\partial_{r}\\
&&+
\left[
(
F\partial_{\theta}W_{1}
)
\partial_{r}F+\Gamma^{2}_{12}
(
F\partial_{\theta}W_{1}
)
F
\right]\partial_{\theta},\\
\nabla_V[V,W]
&=&
\left[
F\partial_{\theta}
(
F\partial_{\theta}W_{1}
)
+
\Gamma^{1}_{22}F
(
F\partial_\theta W_{2}
-
W_{1}\partial_{r}F
)
\right]\partial_r\\
&&+
\left[
F\partial_{\theta}
(
F\partial_\theta W_{2}
-
W_{1}\partial_{r}F
)
+
\Gamma^{2}_{21}F
(
F\partial_{\theta}W_{1}
)
\right]\partial_\theta.
\end{eqnarray*}
By $\partial_{\theta} F=0$,
we have
\begin{eqnarray}
\label{VWVWV}
&&
\nabla_{[V,W]}V+\nabla_V[V,W]
\\
&=&
\left[
F^{2}
\left(
\partial_{\theta}^{2}W_{1}
+
2\Gamma^{1}_{22}
\partial_{\theta}W_{2}
\right)
-
2\Gamma^{1}_{22}
W_{1}
F\partial_{r}F
\right]\partial_r\nonumber\\
&&+
\left[
F^{2}
\left(
\partial_{\theta}^{2}W_{2}
+2\Gamma^{2}_{21}\partial_{\theta} W_{1}
\right)
\right]\partial_\theta.
\nonumber
\end{eqnarray}
Then,
\eqref{integration}
and
\eqref{Christoffel} imply
\begin{eqnarray*}
MC_{V,W}&=&
\int_{M}
\bigg(
F^{2}
c_{1}
W_{1}
\partial_{\theta}^{2}W_{1}
-
2
F^{2}
c_{1}^{2}
\partial_{r}c_{1}
W_{1}
\partial_{\theta}W_{2}
+
c_{1}^{2}\partial_{r}c_{1}
W_{1}^{2}
\partial_{r}(F^{2})
\\
&&+
F^{2}
c_{1}^{3}
W_{2}
\partial_{\theta}^{2}W_{2}
+2F^{2}
c_{1}^{2}
\partial_{r}c_{1}
W_{2}
\partial_{\theta} W_{1}
\bigg)
drd\theta.
\end{eqnarray*}
We note that $F=F(r)$ and $c_{1}=c_{1}(r)$ are independent of the variable $\theta$.
Thus,
applying Stokes theorem to
the first, fourth,
and fifth terms,
we have
\begin{eqnarray*}
&=&
\int_{M}
\bigg(-
F^{2}
c_{1}
\left(
\partial_{\theta}W_{1}
\right)^{2}
-
4
F^{2}
c_{1}^{2}
\partial_{r}c_{1}
W_{1}
\partial_{\theta}W_{2}
+
c_{1}^{2}\partial_{r}c_{1}
\partial_{r}(F^{2})
W_{1}^{2}
\\
&&-
F^{2}
c_{1}^{3}
\left(
\partial_{\theta}W_{2}
\right)^{2}
\bigg)
drd\theta.
\end{eqnarray*}
Recall that 
\begin{eqnarray*}
{\rm div}\: W
&=&
\partial_{r}W_{1}+c_{1}^{-1}\partial_{r}c_{1}W_{1}
                    +\partial_{\theta}W_{2},
\end{eqnarray*}
which implies 
$\partial_{\theta}W_{2}
=-\partial_{r}W_{1}-c_{1}^{-1}\partial_{r}c_{1}W_{1}$
by the assumption
${\rm div} \:W=0$.
Therefore,
we have
\begin{eqnarray*}
&=&
\int_{M}
\bigg(-
F^{2}
c_{1}
\left(
\partial_{\theta}W_{1}
\right)^{2}
+
4
F^{2}
c_{1}^{2}
\partial_{r}c_{1}
W_{1}
\partial_{r}W_{1}\\
&&+
4
F^{2}
c_{1}
\left(
\partial_{r}c_{1}\right)^{2}
W_{1}^{2}
+
c_{1}^{2}\partial_{r}c_{1}
\partial_{r}(F^{2})
W_{1}^{2}
\\
&&\quad -
F^{2}
c_{1}^{3}
\left(
\partial_{r}W_{1}
\right)^{2}
-
2
F^{2}
c_{1}^{2}
\partial_{r}c_{1}
W_{1}
\partial_{r}
W_{1}
-
F^{2}
c_{1}
(
\partial_{r}c_{1})^{2}
W_{1}^{2}
\bigg)
drd\theta.
\end{eqnarray*}
This is equal to
\begin{eqnarray*}
&=&
\int_{M}
\bigg(-
F^{2}
c_{1}
\left(
\partial_{\theta}W_{1}
\right)^{2}
-
F^{2}
c_{1}^{3}
\left(
\partial_{r}W_{1}
\right)^{2}
\\
&&+
\left(
2
F^{2}
c_{1}^{2}
\partial_{r}c_{1}
\right)
W_{1}
\partial_{r}
W_{1}
\\
&&\quad+
\Big(
3
F^{2}
c_{1}
\left(\partial_{r}c_{1}
\right)^{2}
+
c_{1}^{2}\partial_{r}c_{1}
\partial_{r}(F^{2})
\Big)
W_{1}^{2}
\bigg)
drd\theta.
\end{eqnarray*}
We note that
the values of
$F^{2}c_{1}^{2}\partial_{r}c_{1}$
at $r=d$ and $r=-d$ are zero
by
$\lim_{r\to d} c_{1}(r)=\lim_{r\to -d} c_{1}(r)=0$.
(The assumption $\dot{c}_{1}(r)^{2}+
\dot{c}_{2}(r)^{2}=1$ implies that
$\partial_{r}c_{1}$ is bounded.)
Thus,
applying the
Stokes theorem to the term
$c_{1}^{2}\partial_{r}c_{1}
\partial_{r}(F^{2})
\left(
W_{1}
\right)^{2}$,
we have
\begin{eqnarray*}
&=&
\int_{M}
\bigg(-
F^{2}
c_{1}
\left(
\partial_{\theta}W_{1}
\right)^{2}
-
F^{2}
c_{1}^{3}
\left(
\partial_{r}W_{1}
\right)^{2}
\nonumber\\
&&+
\left(
2
F^{2}
c_{1}^{2}
\partial_{r}c_{1}
-2F^{2}c_{1}^{2}\partial_{r}c_{1}
\right)
W_{1}
\partial_{r}
W_{1}
\nonumber\\
&&\quad+
\Big(
3
F^{2}
c_{1}
\left(\partial_{r}c_{1}
\right)^{2}
-
2
F^{2}
c_{1}(\partial_{r}c_{1})^{2}
-
F^{2}
c_{1}^{2}\partial_{r}^{2}c_{1}
\Big)
W_{1}^{2}
\bigg)
drd\theta
\nonumber\\
&=&
\int_{M}
F^{2}c_{1}
\bigg(-
\left(
\partial_{\theta}W_{1}
\right)^{2}
-
c_{1}^{2}
\left(
\partial_{r}W_{1}
\right)^{2}
+
\Big(
\left(\partial_{r}c_{1}
\right)^{2}
-
c_{1}\partial_{r}^{2}c_{1}
\Big)
W_{1}^{2}
\bigg)
drd\theta.
\end{eqnarray*}
This completes the proof.
\end{proof}
Recall that
$MC_{V,W}:=(\nabla_{[V,W]}V+\nabla_{V}[V,W], W)$.
For the existence of $W\in T_{e}{\mathcal D}^{s}_{\mu}(M)$
satisfying $MC_{V,W}>0$,
we have the following:
\begin{prop}
\label{epsilon}
Suppose
$s>3$ and
$
\left(
\partial_{r}c_{1}
\right)^{2}
-
c_{1}\partial_{r}^{2}c_{1}>1$.
Then 
for any
zonal flow
$
V\in T_{e}{\mathcal D}^{s}_{\mu}(M)
$
whose support is contained in $R'(c)$
(see \eqref{def R'} for the definition),
there exists $W_{0}\in T_{e}{\mathcal D}^{s}_{\mu}(M)$
satisfying
$MC_{V,W_{0}}>0$.
\end{prop}

\begin{remark}
We can easily relax the condition on $V$.
However we omit its detail here, since we would like to keep the simple statement.
\end{remark}

\begin{proof}
Set $\epsilon(r):=\sqrt{\left(
\partial_{r}c_{1}
\right)^{2}
-
c_{1}\partial_{r}^{2}c_{1}-1}$
and write $V=F(r)\partial_\theta$.
The assumption of the support of $V$
implies
that
the support of $F$ is properly contained in $I_{d}$.
Define 
a 
divergence-free 
vector field $W_{0}=
W_{01}\partial_{r}+
W_{02}\partial_{\theta}$
on $I_{d}\times S^{1}$
by
\begin{eqnarray}
W_{0}:=
h(r)\sin\theta\partial_{r}
+
\left(
\partial_{r}h(r)+\frac{h(r)\partial_{r}c_{1}(r)}
{c_{1}(r)}
\right)\cos\theta\partial_{\theta}
\label{def W0}
\end{eqnarray}
for some smooth real valued function $h=h(r)$
on $r\in I_{d}$.
Moreover,
by Proposition \ref{formula},
we have
\begin{eqnarray}
&&
MC_{V,W_{0}}
\label{MCW0}
\\
&=&
\int_{-d}^{d}
\int_{-\pi}^{\pi}
-
F^{2}
c_{1}
\bigg(
\left(
\partial_{\theta}W_{01}
\right)^{2}
+
c_{1}^{2}
\left(
\partial_{r}W_{01}
\right)^{2}
-
\Big(
1+\epsilon^{2}(r)
\Big)
W_{01}^{2}
\bigg)
d\theta dr
\nonumber
\\
&=&
\int_{-d}^{d}
-F^{2}
c_{1}
\int_{-\pi}^{\pi}
\bigg(
h^{2}
(\cos^{2}\theta
-
\sin^{2}\theta
-\epsilon^{2}
\sin^{2}\theta
)
+
c_{1}^{2}
(
\partial_{r}h
)^{2}
\sin^{2}\theta
\bigg)
d\theta dr
\nonumber
\\
&=&
\pi
\int_{-d}^{d}
F^{2}
c_{1}
\bigg(
h^{2}
\epsilon^{2}
-
c_{1}^{2}
(
\partial_{r}h
)^{2}
\bigg)
dr.
\nonumber
\end{eqnarray}
Because the support of $F$
is
contained in $I_{d}$,
there exists a
smooth real valued function $h=h(r)$
on $r\in I_{d}$
satisfying
the following properties:
\begin{enumerate}
\item
$\partial_{r}h=0$
on the support of $F$,
\item
$h\neq0$
on the support of $F$,
\item
$h$ is identically zero near the points
$r=-d$ and $r=d$.
\end{enumerate}
For such $h$,
the last term of 
\eqref{MCW0}
is positive
and $W_{0}$ defines
an element of
$T_{e}{\mathcal D}^{s}_{\mu}(M)$.
This completes the proof.
\end{proof}

\begin{remark}
\label{many}
The proof of Proposition \ref{epsilon}
implies 
that
$\#\{
W\in V^{\perp}
\mid
MC_{V,W}=MC_{V,W_{0}},\:
|W|=1
\}=\infty
$.
\end{remark}

\begin{cor}
\label{corconj}
Suppose
that
$s>3$ 
and
$
\left(
\partial_{r}c_{1}
\right)^{2}
-
c_{1}\partial_{r}^{2}c_{1}>1$.
Then 
for any
zonal flow $V\in T_{e}{\mathcal D}^{s}_{\mu}(M)$
whose support is contained in $R'(c)$,
there exists a point conjugate
to $e\in{\mathcal D}^{s}_{\mu}(M)$ along $\eta(t)=\widetilde{\exp}_{e}(tV)$ on $0\leq t\leq t^{*}$
for some $t^{*}>0$.
\end{cor}
\begin{proof}
It is obvious 
by Fact \ref{factM-criterion}
and
Proposition
\ref{epsilon}.
\end{proof}

\section{The main theorems: ellipsoid and sphere cases}
\label{ellipsoid}
In this section,
we investigate the case that $M$ is a 2-dimensional ellipsoid
and the case
$M$ is a sphere,
more precisely,
for $M=M_{a}:=\{(x,y,z)\in{\mathbb R}^{3}
\mid
x^{2}+y^{2}=a^{2}(1-z^{2})
\}$, 
the case of $a>1$ (having a bulge around its equatorial middle and is flattened at the poles),
and the case of $a=1$ (sphere).

Let $E_{a}:=\{(x,z)\in{\mathbb R}^{2}
\mid
x^{2}=a^{2}(1-z^{2})
\}$ be a ellipse in ${\mathbb R}^{2}$
and $\ell$  the arc length of $E_{a}$.
Set
$d:=\ell/4$
and
take a curve
$$c(r):=(c_{1}(r),c_{2}(r)):I_{d}=(-d,d)\to E_{a}$$
satisfying
$\lim_{r\to-d}c(r)=(0,-1)$,
$\lim_{r\to d}c(r)=(0,1)$,
$c_{1}(r)>0$ 
and
$\dot{c}_{1}(r)^{2}+\dot{c}_{1}(r)^{2}=1$
on $r\in I_{d}$.
Then,
we have $M_{a}=R(c)$
(see Lemma \ref{R(c)}).
We note that $c_{1}(r)$
is a positive even function 
by the definition.

Therefore, 
we can apply the results of Section \ref{rotationally}
to the ellipsoid case.
For this purpose,
we firstly show the following:
\begin{prop}
\label{ellipsoid-pos}
If $a>1$,
then
$
(\dot{
c}_{1}
)^{2}
-
c_{1}\ddot{c}_{1}-1>0$.
\end{prop}
\begin{remark}
	In contrast to this,
	we have $(\dot{c})^{2}-c_{1}\ddot{c}_{1}\equiv 1$ in the case of $a=1$
	(i.e., sphere case).
\end{remark}
\begin{proof}
Recall that
$
E_{a}:=\{(x,z)\in{\mathbb R}^{2}
\mid
x^{2}=a^{2}(1-z^{2})
\}$.
We 
note that
the gradient of
the function
$
x^{2}-a^{2}(1-z^{2})$ on ${\mathbb R}^{2}$
is equal to
$
2x\partial_{x}+2a^{2}z\partial_{z}$.
Therefore
$
x\partial_{x}+a^{2}z\partial_{z}$
is
a normal vector field of $E_{a}$.
Thus
$-a^{2}z\partial_{x}+x\partial_{z}$
is tangent to $E_{a}$.
This implies
\begin{eqnarray*}
(\dot{c}_{1},\dot{c}_{2})=
\frac{1}{\sqrt{c_{1}^{2}+a^{4}c_{2}^{2}}}(-a^{2}c_{2},c_{1}).
\end{eqnarray*}
Thus we have
\begin{eqnarray*}
\ddot{c}_{1}
&=&
\frac{-a^{2}}
{\sqrt{c_{1}^{2}+a^{4}c_{2}^{2}}}\dot{c}_{2}
+
(-a^{2}c_{2})
\left(-\frac{1}{2}\right)
\frac{2c_{1}\dot{c}_{1}+2a^{4}c_{2}\dot{c}_{2}}
{(c_{1}^{2}+a^{4}c_{2}^{2})^{\frac{3}{2}}}\\
&=&
\frac{-a^{2}c_{1}^{3}-a^{4}c_{1}c_{2}^{2}}
{(c_{1}^{2}+a^{4}c_{2}^{2})^{2}}.
\end{eqnarray*}
Therefore 
\begin{eqnarray*}
(\dot{
c}_{1}
)^{2}
-
c_{1}\ddot{c}_{1}-1
&=&
\frac{
(a^{2}-1)c_{1}^{4}
}
{(c_{1}^{2}+a^{4}c_{2}^{2})^{2}}.
\end{eqnarray*}
This and the assumption $a>1$
imply the proposition.
\end{proof}
We now recall the first main theorem:
\let\temp\thetheorem
\renewcommand{\thetheorem}{\ref{main-theorem}}
\begin{theorem}
Let $s>3$ and  $M_{a}:=\{(x,y,z)\in{\mathbb R}^{3}
\mid
x^{2}+y^{2}=a^{2}(1-z^{2})
\}$
be an ellipsoid with $a>1$.
For any
zonal flow
$V\in T_{e}{\mathcal D}^{s}_{\mu}(M)$ whose support is contained in $M_{a}\backslash \{(0,0,1),(0,0,-1)\}$,
there exists
$W\in T_{e}{\mathcal D}^{s}_{\mu}(M)$
satisfying $MC_{V,W}>0$.
\end{theorem}
\let\thetheorem\temp
\addtocounter{theorem}{-1}
\begin{proof}
This is a consequence of
Corollary \ref{corconj}
and
Proposition \ref{ellipsoid-pos}.
\end{proof}
Now we investigate the case that $M$ is a 2-dimensional sphere,
namely, the
case of $a=1$.
Therefore we have
$M:=M_{1}=\{(x,y,z)\in{\mathbb R}^{3}
\mid
x^{2}+y^{2}=(1-z^{2})
\}=S^{2}$,
$d=\frac{\pi}{2}$
and
$c_{1}(r):=\cos r$.
By Proposition \ref{formula},
we have
\begin{eqnarray*}
MC_{V,W}
=
\int_{-\pi/2}^{\pi/2}
\int_{-\pi}^{\pi}
-
F^{2}
c_{1}
\bigg(
\left(
\partial_{\theta}W_{1}
\right)^{2}
+
c_{1}^{2}
\left(
\partial_{r}W_{1}
\right)^{2}
-
W_{1}^{2}
\bigg)
d\theta dr
\end{eqnarray*}
for $V=F\partial_{\theta}$ and
$W=W_{1}\partial_{r}+W_{2}\partial_{\theta}$.
Also we now recall the second main theorem:
\let\temp\thetheorem
\renewcommand{\thetheorem}{\ref{second-main-theorem}}
\begin{theorem}
Suppose $s>3$.
For any zonal flow $V\in T_{e}{\mathcal D}^{s}_{\mu}(S^2)$
and
any $W\in T_{e}{\mathcal D}^{s}_{\mu}(S^{2})$,
we have
$MC_{V,W}\leq 0$.
\end{theorem}
\let\thetheorem\temp
\addtocounter{theorem}{-1}
\begin{proof}
By Sobolev embedding theorem,
$W_{1}$ and $W_{2}$
are of class $C^{2}$ (see Remark \ref{reason for C^2}).
Thus,
we 
can consider the Fourier series of $
W_{j}(r,\theta):=\sum_{k\in{\mathbb Z}}w_{j}^{(k)}(r)e^{ik\theta}$
for $j\in\{1,2\}$,
where
$w_{j}^{(k)}(r)=\int_{-\pi}^{\pi}
W_{j}(r,\theta)e^{-ik\theta}d\theta$.
By Green-Stokes theorem and $\operatorname{div}W=0$,
we have
\begin{eqnarray*}
	c_{1}(r)w_{1}^{(0)}(r)
	&=&
	\int_{-\pi}^{\pi}
	c_{1}(r)
W_{1}(r,\theta)d\theta
\\
&=&
\int_{r}^{\pi/2}
\int_{-\pi}^{\pi}
c_{1}(r)
\operatorname{div}W(r',\theta)
d\theta dr'
\\
&=&
0,
\end{eqnarray*}
which implies $w_{1}^{(0)}\equiv 0$.
Note that the complex conjugate of $w_{1}^{(k)}$ is equal to $w_{1}^{(-k)}$ because $W_{1}$ is a real valued function.
Then, 
\begin{eqnarray*}
&&W_{1}^{2}
-
\left(
\partial_{\theta}W_{1}\right)^{2}\\
&=&
\sum_{k\in{\mathbb Z}}
\sum_{n+m=k}(1+nm)w_{1}^{(n)}w_{1}^{(m)}e^{ik\theta}\\
&=&
\sum_{k\neq 0}
\sum_{n+m=k}
(1+nm)w_{1}^{(n)}w_{1}^{(m)}e^{ik\theta}
+
\sum_{l\neq 0}
(1-l^{2})|w_{1}^{(l)}|^2.
\end{eqnarray*}
Therefore,
we have
\begin{eqnarray*}
\frac{1}{2\pi}
\int_{-\pi}^{\pi}
\bigg(
W_{1}^{2}
-
\left(
\partial_{\theta}W_{1}\right)^{2}
\bigg)
d\theta
&=&
\sum_{l\neq 0}
(1-l^{2})|w_{1}^{(l)}|^2
\leq 0.
\end{eqnarray*}
Then 
\begin{eqnarray*}
&&
\int_{-\pi/2}^{\pi/2}
\int_{-\pi}^{\pi}
-
F^{2}
c_{1}
\bigg(
\left(
\partial_{\theta}W_{1}
\right)^{2}
+
c_{1}^{2}
\left(
\partial_{r}W_{1}
\right)^{2}
-
W_{1}^{2}
\bigg)
d\theta dr\\
&\leq&
\int_{-\pi/2}^{\pi/2}
\int_{-\pi}^{\pi}
-
F^{2}
c_{1}
\bigg(
\left(
\partial_{\theta}W_{1}
\right)^{2}
-
W_{1}
^{2}
\bigg)
d\theta dr
\leq 0.
\end{eqnarray*}
This completes the proof.
\end{proof}

\begin{remark}
	A geometric meaning of Theorems 
	\ref{main-theorem}
	and
	\ref{second-main-theorem}
	is the following:
	Let $V=F(r)\partial_{\theta}$
	be a zonal flow
	and
	$\eta(t)$
	a corresponding geodesic on	
	${\mathcal D}^{s}_{\mu}(M_{a})$.
	Then,
	its
	length function is 
	\begin{eqnarray*}
	E(\eta)_{0}^{t_{0}}
	&=&
	2\pi t_{0}
	\int_{-\pi}^{\pi}
	F^{2}(r)c_{1}^{3}(r)
	dr
	\end{eqnarray*}
	by 
	\eqref{Eint} and
	\eqref{integration}.
	We note that
	the value
	$2\pi c_{1}^{2}(r)$
	is equal to the length of
	the horizontal circle
	$M_{a}^{r}:=\{(x,y,z)\in M_{a}\mid 
	      z=c_{2}(r)\}$.
	      
	On the other hand,
	if the horizontal circle $M_{a}^{r}$
	is slightly tilted 
	in such a way that the area of
	the enclosed region of $M_{a}^{r}$ is invariant,
	the length of $M_{a}^{r}$
	can be smaller (resp. greater) in the $a>1$ 
	(resp. $a<1$) case
	by a comparison theorem.
	
	For $W\in T_{e}{\mathcal D}^{s}_{\mu}(M_{a})$,
	we have $\exp(tW)\in {\mathcal D}^{s}_{\mu}(M_{a})$,
	namely,
	$\exp(tW)$ preserves the volume element $\mu$ of $M_{a}$.
	Thus,
	a deformation of $W$
	preserves the are of 
	the enclosed region of $M_{a}^{r}$,
	which implies that
	the second variation of $E(\eta)$ by
	$W$
	can be negative in the $a>1$ case.
\end{remark}

\section{Appendix 1: Existence of a conjugate point and the M-criterion}
\label{2D-case}
In Section \ref{rotationally},
it is observed that there are many $W\in T_{e}{\mathcal D}^{s}_{\mu}(M)$
satisfying $MC_{V,W}>0$ for some fixed zonal flow $V$
(see Proposition \ref{epsilon} and Remark \ref{many}),
where $M$ is a compact $2$-dimensional rotationally symmetric manifold.
Therefore,
it seems to be worthwhile 
clarify more 
the meaning of
$W\in T_{e}{\mathcal D}^{s}_{\mu}(M)$
satisfying $MC_{V,W}>0$
in the case that 
$\dim M=2$.
This is the main purpose of this section.
Moreover,
for the completeness,
we also give a proof of the M-criterion (Fact \ref{factM-criterion}) in the 2D case,
which is essentially already proved by Misio{\l}ek.
We suppose that
$M$ is a compact 2-dimensional 
Riemannian manifold
without boundary
in this section.

For a positive number $t_{0}>0$,
we define a subspace $K_{\eta}^{t_{0}}$ of $T_{e}{\mathcal D}^{s}_{\mu}(M)$ by
\begin{eqnarray*}
K_{\eta}^{t_{0}}
:=
\sum_{t\in[0,t_{0}]} 
{\rm Ker}
\left(
T_{tV}\widetilde{\exp}_{e}:
T_{tV}(T_{e}{\mathcal D}^{s}_{\mu}(M))
\simeq
T_{e}{\mathcal D}^{s}_{\mu}(M)
\to
T_{\eta(t)}{\mathcal D}^{s}_{\mu}(M)
\right).
\end{eqnarray*}
We write
$K_{\eta}^{t_{0},\perp}$ is the orthogonal complement of 
$K_{\eta}^{t_{0}}$ with respect to 
the Sobolev inner product,
namely,
the inner product
defines the original topology of $T_{e}{\mathcal D}^{s}_{\mu}(M)$.
In particular, 
$K_{\eta}^{t_{0},\perp}$ is closed in 
$T_{e}{\mathcal D}^{s}_{\mu}(M)$
with respect to the original topology.
We define 
a subset of ${\mathcal D}^{s}_{\mu}(M)$ by
\begin{eqnarray}
\label{Eperp}
E_{\eta}^{t_{0},\perp}:=\widetilde{\exp}_{e}(K_{\eta}^{t_{0},\perp})
\quad \subset {\mathcal D}^{s}_{\mu}(M).
\end{eqnarray}
The
finite-dimensionality of $K^{t_{0}}_{\eta}$ 
and
finite-codimensionlity of $K^{t_{0},\perp}_{\eta}$
in the 2D case
follow from Facts \ref{Fredholm} and \ref{exp2D}.

\begin{fact}[{\cite[Lemma 3]{Mexp2D}}]
\label{exp2D}
Let $M$ be a compact 2-dimensional Riemannian manifold without boundary.
Then 
any finite geodesic segment in ${\mathcal D}^{s}_{\mu}(M)$ contains at most finitely many conjugate
points.
\end{fact}
\begin{remark}
Fact \ref{exp2D} implies that for any $t_{0}>0$,
there exist $N\in{\mathbb N}$ and
$t_{1},\dots,t_{N}\in[0,t_{0}]$
such that
$\eta(t_{1}),\dots,\eta(t_{N})$ exhaust 
all points conjugate to $e\in{\mathcal D}^{s}_{\mu}(M)$
along $\eta(t)$ for $0\leq t\leq t_{0}$.
Then we have
$$
K_{\eta}^{t_{0}}=
\bigoplus_{j=1}^{N} 
{\rm Ker}
\left(
T_{t_{j}V}\widetilde{\exp}_{e}:
T_{t_{j}V}(T_{e}{\mathcal D}^{s}_{\mu}(M))
\simeq T_{e}{\mathcal D}^{s}_{\mu}(M)
\to
T_{\eta(t_{j})}{\mathcal D}^{s}_{\mu}(M)
\right).
$$
\end{remark}
\begin{lem}
\label{diffeo}
Let $M$ be a compact 2-dimensional Riemannian manifold without boundary.
Then
for any $t\in[0,t_{0}]$,
we have an isomorphism
$$
T_{tV}(K_{\eta}^{t_{0},\perp})
\DistTo
T_{tV}\widetilde{\exp}_{e}
(
T_{tV}(K_{\eta}^{t_{0},\perp})
),
$$
which is induced by
$
T_{tV}\widetilde{\exp}_{e}:
T_{tV}(T_{e}{\mathcal D}^{s}_{\mu}(M))
\to
T_{\eta(t)}{\mathcal D}^{s}_{\mu}(M).
$
\end{lem}
\begin{proof}
Recall that
$T_{tV}\widetilde{\exp}_{e}:
T_{tV}(T_{e}{\mathcal D}^{s}_{\mu}(M))
\to
T_{\eta(t)}{\mathcal D}^{s}_{\mu}(M)$
is a nonlinear Fredholm map by
Fact \ref{Fredholm}.
In particular,
it has a closed range,
namely,
${\rm Image}(T_{tV}\widetilde{\exp}_{e})$
is a closed subspace of $T_{\eta(t)}{\mathcal D}^{s}_{\mu}(M)$.
Then,
we have an isomorphism
\begin{equation}
({\rm Ker}(T_{tV}\widetilde{\exp}_{e}))^{\perp}
\DistTo
T_{tV}\widetilde{\exp}_{e}(({\rm Ker}(T_{tV}\widetilde{\exp}_{e}))^{\perp})
\quad
\overset{closed}{
\subset}
T_{\eta(t)}{\mathcal D}^{s}_{\mu}(M)
\label{Ker=Tv}
\end{equation}
by the open mapping theorem and the following diagram:
$$
\begin{array}{ccccccc}
T_{tV}\widetilde{\exp}_{e}
&
:
&
T_{tV}(T_{e}{\mathcal D}^{s}_{\mu}(M))
&
\to
&
T_{\eta(t)}{\mathcal D}^{s}_{\mu}(M)
\\
&
&
\rotatebox{90}{$=$}
&
&
\rotatebox{90}{$\subset$}
&
&
\\
&
&
\left({\rm Ker}(T_{tV}\widetilde{\exp}_{e})
\right)^{\perp}
&
\DistTo
&
T_{tV}\widetilde{\exp}_{e}
(
\left({\rm Ker}(T_{tV}\widetilde{\exp}_{e})
\right)^{\perp}
)
&
=
&
{\rm Image}(T_{tV}\widetilde{\exp}_{e})
\\
&
&
\oplus
&
&
&
&
\\
&
&
{\rm Ker}(T_{tV}\widetilde{\exp}_{e})
&
\to
&
0
&
&
\end{array}
$$

Next,
we show that 
$K^{t_{0},\perp}_{\eta}\subset T_{e}{\mathcal D}^{s}_{\mu}(M)$
satisfies
the following properties
for any $t\in[0,t_{0}]$:
\begin{enumerate}
\item
\label{closed}
$T_{tV}(K^{t_{0},\perp}_{\eta})\simeq K^{t_{0},\perp}_{\eta}$ is a closed subspace of $T_{tV}(T_{e}{\mathcal D}^{s}_{\mu}(M)) \simeq T_{e}{\mathcal D}^{s}_{\mu}(M)$,

\item
$T_{tV}(K^{t_{0},\perp}_{\eta})$
is contained in 
$\left({\rm Ker}(T_{tV}\widetilde{\exp}_{e})
\right)^{\perp}$.
\end{enumerate}

Indeed,
(i) follows from the fact that 
$K^{t_{0},\perp}_{\eta}$
is the orthogonal complement of $K^{t_{0}}_{\eta}$
and
(ii)
is a consequence of the definition of 
$K^{t_{0},\perp}_{\eta}$.
The following diagram describes the relationship among regarding spaces:
$$
\begin{array}{ccccccc}
T_{tV}\widetilde{\exp}_{e}
&
:
&
T_{tV}(T_{e}{\mathcal D}^{s}_{\mu}(M))
&
&
&
\to
&
T_{\eta(t)}{\mathcal D}^{s}_{\mu}(M)
\\
&
&
\rotatebox{90}{$=$}
&
&
&
&
\rotatebox{90}{$\subset$}
\\
&
&
\left({\rm Ker}(T_{tV}\widetilde{\exp}_{e})
\right)^{\perp}
&
\overset{closed}{\supset}
&
T_{tV}(
K_{\eta}^{t_{0},\perp})
&
\DistTo
&
T_{tV}\widetilde{\exp}_{e}
(
T_{tV}(K_{\eta}^{t_{0},\perp})
)
\\
&
&
\oplus
&
&
&
&
\\
&
&
{\rm Ker}(T_{tV}\widetilde{\exp}_{e})
&
\subset
&
T_{tV}(
K_{\eta}^{t_{0}})
&
\to
&
0
\end{array}
$$
Therefore,
the isomorphism \eqref{Ker=Tv}
induces the desired isomorphism.
\end{proof}
\begin{remark}
This lemma is not true in the case that $\dim M=3$,
see \cite[Section 4]{EMis}.
\end{remark}

Recall that
we say
$\xi(r,t):(-\varepsilon,\varepsilon)\times[0,t_{0}]\to {\mathcal D}^{s}_{\mu}(M)$
is
a variation of a geodesic $\eta(t)$ on ${\mathcal D}^{s}_{\mu}(M)$
with fixed endpoints,
if it satisfies
$\xi(r,0) \equiv \eta(0)$, $\xi(r,t_{0}) \equiv \eta(t_{0})$ and $\xi(0,t)=\eta(t)$.
We sometimes write $\xi_{r}(t)$ for $\xi(r,t)$.

\begin{prop}
\label{pos}
Let 
$M$ be a compact 2-dimensional Riemannian manifold without boundary,
$V\in T_{e}{\mathcal D}^{s}_{\mu}(M)$ a time independent solution of Euler equations \eqref{Epre}
on $M$
and
$\eta(t)$ the geodesic on ${\mathcal D}^{s}_{\mu}(M)$ corresponding to $V$.
Let
$\xi(r,t):(-\varepsilon,\varepsilon)\times[0,t_{0}]\to {\mathcal D}^{s}_{\mu}(M)$ be a variation of $\eta(t)$
with fixed endpoints
satisfying ${\rm Image}\:(\xi)\subset E_{\eta}^{t_{0},\perp}$.
Then 
we have
$E''(\eta)^{t_0}_{0}(X,X)\geq 0$,
where $X=\partial_{r}\xi(r,t)|_{r=0}$.
\end{prop}

\begin{proof}
We almost follow the same strategy  in \cite[Lemma 3]{MconjT2}.

Lemma \ref{diffeo} implies
that
there exists a sufficiently small open neighborhood $U_{t}\subset K_{\eta}^{t_{0},\perp}$ of $tV$ such that 
$U_{t}$ is diffeomorphic to 
$\widetilde{\exp}_{e}(U_{t})\subset E_{\eta}^{t_{0},\perp}=\widetilde{\exp}_{e}(K_{\eta}^{t_{0},\perp})$
by the inverse function theorem
(See \cite[Proposition 2.3]{Lang}, for instance)
for any $t\in[0,t_{0}]$.
In particular,
$\widetilde{\exp}_{e}(U_{t})$ is open in $E_{\eta}^{t_{0},\perp}$
and
we can define $\widetilde{\log}_{e}:=\widetilde{\exp}_{e}^{-1}: \widetilde{\exp}_{e}(U_{t})\to U_{t}$.
Set $U:=\bigcup_{t\in[0,t_{0}]}  U_{t}$, then we have 
$tV\in U$ for any $t\in[0,t_{0}]$ because $tV\in U_{t}\subset U$.
Thus, we have $\eta(t)=\widetilde{\exp}_{e}(tV)\in \widetilde{\exp}_{e}(U)$,
namely,
$\xi(0,t)\in \widetilde{\exp}_{e}(U)$.
Then,
we can assume ${\rm Image}\:(\xi) \subset \widetilde{\exp}_{e}(U)$
by taking smaller $\varepsilon>0$
because
$\widetilde{\exp}_{e}(U)$ is open in  $E_{\eta}^{t_{0},\perp}$
and
${\rm Image}\:(\xi)$ is contained in $E_{\eta}^{t_{0},\perp}$
by the assumption.
$$
\begin{array}{ccccccc}
T_{tV}\widetilde{\exp}_{e}
&
:
&
T_{tV}(K_{\eta}^{t_{0},\perp})
&
\DistTo
&
T_{tV}\widetilde{\exp}_{e}
(
T_{tV}(K_{\eta}^{t_{0},\perp})
)
\\
\\
\widetilde{\exp}_{e}
&
:
&
\:\:
K^{t_{0},\perp}_{\eta}
&
\to
&
\widetilde{\exp}_{e}(
K^{t_{0},\perp}_{\eta}
)
&
=
&
E^{t_{0},\perp}_{\eta}
\\
&
&
\rotatebox{90}{$\subset$}
&
&
\rotatebox{90}{$\subset$}
\\
&
&
U_{t}
&
\DistTo
&
\widetilde{\exp}_{e}(U_{t})
\\
&
&
\rotatebox{90}{$\in$}
&
&
\rotatebox{90}{$\in$}
\\
&
&
tV
&
\mapsto
&
\widetilde{\exp}_{e}(
tV
)
&
=
&
\eta(t)
\end{array}
$$
Therefore
we can define a curve $c_{r}(t):=\widetilde{\log}_{e} \xi_{r}(t)$ in $T_{e}{\mathcal D}^{s}_{\mu}(M)$
and 
$\ell_{r}(t):=|c_{r}(t)|=\sqrt{(c_{r}(t),c_{r}(t))}$.
Then we have 
$\ell_{r}(0)=0$, $\ell_{r}(t_{0})=t_{0}|V|$
and $c_{r}(t)=\ell(t)\frac{c_{r}(t)}{|c_{r}(t)|}$.
Thus, we obtain
\begin{eqnarray*}
\dot{c}_{r}(t)=
\dot{\ell}_{r}(t)
\frac{c_{r}(t)}{|c_{r}(t)|}
+
\ell_{r}(t)
\frac{d}{dt}
\left(
\frac{c_{r}(t)}{|c_{r}(t)|}
\right).
\end{eqnarray*}
Then,
for any $r\in(-\varepsilon,\varepsilon)$,
we have
\begin{eqnarray*}
|\dot{\xi}_{r}(t)|
=
\left|
\frac{d}{dt}
\left(
\widetilde{\rm exp}_{e}
c_{r}(t)
\right)
\right|
=
\left|
T_{c_{r}(t)}\widetilde{\rm exp}_{e}(\dot{c}_{r}(t))
\right|
=
\left|
\dot{c}_{r}(t)
\right|
\geq
\dot{\ell}_{r}(t)^{2}.
\end{eqnarray*}
In the third equality, 
we used Gauss's lemma or 
\cite[Lemma 2]{Mexp2D}.
Then,
by \eqref{Eint}
and
the Cauchy-Schwartz inequality,
we have
\begin{eqnarray*}
E(\xi_{r})
&\geq&
\frac{1}{2}\int_{0}^{t_{0}}
\dot{\ell}_{r}(t)^{2}dt
=
\frac{1}{2t_{0}}
\left(
\int_{0}^{t_{0}}
\dot{\ell}_{r}(t)^{2}dt
\right)
\left(
\int_{0}^{t_{0}}
1^{2}
dt
\right)\\
&\geq&
\frac{1}{2t_{0}}
\left(
\int_{0}^{t_{0}}
\dot{\ell}_{r}(t)dt
\right)^{2}
=\frac{t_{0}}{2}|V|^{2}\\
&=&
E(\eta)
\end{eqnarray*}
for any $r\in(-\varepsilon,\varepsilon)$.
This implies $E''(\eta)^{t_0}_{0}(X,X)\geq 0$.
\end{proof}
Recall that
\begin{eqnarray*}
t_{V,W,k}:=\pi|W|\sqrt{\frac{k}{MC_{V,W}}},\quad
f_{V,W,k}(t):=\sin \left(
\frac{t}{|W|}
\sqrt{\frac{MC_{V,W}}{k}}\right),\\
{\widetilde W}^{k}_{\eta(t)}:=
f_{V,W,k}(t)
(W\circ \eta(t))\in
T_{\eta(t)}{\mathcal D}^{s}_{\mu}(M)
\end{eqnarray*}
for $W\in T_{e}{\mathcal D}^{s}_{\mu}(M)$ 
satisfying $MC_{V,W}>0$
and
$k\in{\mathbb R}_{>0}$.
\begin{cor}
\label{subnotE}
Let $M$ be a compact $n$-dimensional Riemannian manifold without boundary
and
$s>2+\frac{n}{2}$.
Suppose that 
$V\in T_{e}{\mathcal D}^{s}_{\mu}(M)$ is a time independent solution of \eqref{Epre}
and that
$W\in T_{e}{\mathcal D}^{s}_{\mu}(M)$
satisfies
$MC_{V,W}>0$.
Take the geodesic $\eta(t)$ on ${\mathcal D}^{s}_{\mu}(M)$ corresponding to $V$
and
define a variation 
$\xi^{k}(r,t):(-\varepsilon,\varepsilon)\times[0,t_{V,W,k}]\to {\mathcal D}^{s}_{\mu}(M)$ of $\eta(t)$ 
with fixed endpoints
by $\xi^{k}_{r}(t):=\widetilde{\exp}_{\eta(t)}(r\widetilde{W}^{k})$.
Then
we have 
$\{
\xi^{k}_{r}(t)\mid t\in[0,t_{V,W,k}]\:|r|\ll1
\}\not\subset E_{\eta}^{t_{V,W,k},\perp}$
for any $k>1$.
\end{cor}
\begin{proof}
Suppose that the contrary,
namely,
$\{
\xi^{k}_{r}(t)\mid t\in[0,t_{W,k}]\:|r|\ll1
\}\subset E_{\eta}^{t_{V,W,k},\perp}$.
Then we have
$E''(\eta)^{t_{V,W,k}}_{0}(\widetilde{W}^{k},\widetilde{W}^{k})\geq 0$
by Proposition \ref{pos}.
However,
this contradicts
Corollary \ref{negative}.
\end{proof}
\begin{cor}
\label{conjugate}
(Existence of a conjugate point, M-critetion)\ 
Let $M$ be a compact $2$-dimensional Riemannian manifold without boundary
and
$s>2+\frac{n}{2}$.
Suppose that
$V \in T_{e}{\mathcal D}^{s}_{\mu}(M)$
is
a time independent solution of Euler equations \eqref{Epre} on $M$.
Take the geodesic $\eta(t)$ on ${\mathcal D}^{s}_{\mu}(M)$ corresponding to $V$.
If there exists a $W_{0}\in T_{e}{\mathcal D}^{s}_{\mu}(M)$
satisfying $MC_{V,W_{0}}>0$,
there exists a point conjugate
to $e\in{\mathcal D}^{s}_{\mu}(M)$ along $\eta(t)$ for $0\leq t\leq t_{V,W_{0},1}$.
\end{cor}
\begin{proof}
Suppose that there are no points conjugate
to $e\in{\mathcal D}^{s}_{\mu}(M)$ along $\eta(t)$ for $0\leq t\leq t_{V,W_{0},k}$
for $k>1$.
Then 
$K_{\eta}^{t_{V,W_{0},k}}=0$
and
$K_{\eta}^{t_{V,W_{0},k},\perp}=T_{e}{\mathcal D}^{s}_{\mu}(M)$.
In particular,
${\rm Image}(\xi^{k}) \subset E_{\eta}^{t_{V,W_{0},k},\perp}=\widetilde{\exp}_{e}(T_{e}{\mathcal D}^{s}_{\mu}(M))$,
where
$\xi^{k}(r,t):=\widetilde{\exp}_{\eta(t)}(r\widetilde{W}^{k})$.
Therefore 
Proposition \ref{pos}
implies
$E''(\eta)^{t_{V,W_{0},k}}_{0}(\widetilde{W}_{0}^{k},\widetilde{W}_{0}^{k})\geq 0$.
On the other hand,
we have
$E''(\eta)^{t_{V,W_{0},k}}_{0}(\widetilde{W}_{0}^{k},\widetilde{W}_{0}^{k})<0$
by Corollary \ref{negative}.
This contradiction implies that
there exists
a point conjugate
to $e\in{\mathcal D}^{s}_{\mu}(M)$ along $\eta(t)$ on $0\leq t\leq t_{V,W_{0},k}$
for any $k>1$.
Taking a limit,
we have the corollary.
\end{proof}
\begin{cor}
Let 
$M$ be a compact 2-dimensional Riemannian manifold without boundary
and
$s>2+\frac{n}{2}$.
Suppose
that
$V\in T_{e}{\mathcal D}^{s}_{\mu}(M)$
is a time independent solution of Euler equations \eqref{Epre} on $M$
and
take
$\eta(t)$ the geodesic on ${\mathcal D}^{s}_{\mu}(M)$ corresponding to $V$.
Then,
$\sup\{MC_{V,W}\mid W\in T_{e}{\mathcal D}^{s}_{\mu}(M)\}<\infty$.
\end{cor}
\begin{proof}
By Fact \ref{exp2D},
there exists $t^{*}>0$ such that 
there are no points conjugate
to $e\in{\mathcal D}^{s}_{\mu}(M)$ along $\eta(t)$ for $0\leq t\leq t^{*}$.
On the other hand,
by Corollary \ref{conjugate},
if $W\in T_{e}{\mathcal D}^{s}_{\mu}(M)$
satisfies
$MC_{V,W}>0$,
there exists a point conjugate to $e\in{\mathcal D}^{s}_{\mu}(M)$
along $\eta(t)$ for $0\leq t\leq t_{V,W,1}$.
Thus we have $t_{V,W,1}=\pi|W|\sqrt{\frac{1}{MC_{V,W}}}>t^{*}>0$.
This implies the corollary.
\end{proof}

\section{Appendix 2: Fredholmness of the exponential map in 2D case}
Let
$M$ be a compact $2$-dimensional Riemannian manifold without boundary and
$s>3$.
Suppose that
$V\in T_{e}{\mathcal D}^{s}_{\mu}(M)$
is a time independent solution of the Euler equations 
\eqref{Epre}
on $M$
and that
$W\in T_{e}{\mathcal D}^{s}_{\mu}(M)$
satisfies $MC_{V,W}>0$.
Take a geodesic $\eta(t)$ on ${\mathcal D}^{s}_{\mu}(M)$ satisfying $V=  \dot{\eta}\circ\eta^{-1}$ as a vector field on $M$.
Then,
Corollary \ref{negative}
implies
that
there exists 
a vector field $\widetilde{W}^{k}$
on ${\mathcal D}^{s}_{\mu}(M)$ along $\eta(t)$
($0\leq t \leq t_{V,W,k}$)
such that
$E''(\eta)^{t_{V,W_{0},k}}_{0}({\widetilde W}^{k},{\widetilde W}^{k})< 0$
for any $k>1$.
It seems that
this contradicts to Fact \ref{exp2D}
because
the dimension of 
the subspace of
the space of all
vector fields along $\eta(t)$
on which
$E''(\eta)^{t_{V,W_{0},k}}_{0}(\cdot,\cdot)$ is negative definite
is equal to the number of conjugate points 
to $e\in{\mathcal D}^{s}_{\mu}(M)$
along $\eta(t)$
(see Fact \ref{It} given below).
However,
this apparent paradox is not really an issue
because
$\{\widetilde{W}^{k}\}$
does not form any infinite-dimensional vector space.
The main purpose of this Appendix
is to explain this phenomena.
For the precise statement,
we fix some notations from now on.
We refer to \cite[Section 8]{MP} for the detail of 
the content of this Appendix.

Let
$M$ be a compact $2$-dimensional Riemannian manifold without boundary.
Take a geodesic $\eta(t)$ ($0\leq t\leq 1$)
joining $e\in {\mathcal D}^{s}_{\mu}(M)$
and
$\eta(1)$
on ${\mathcal D}^{s}_{\mu}(M)$.
For $0\leq t\leq 1$,
we consider the space $C^{\infty}_{t}:=C^{\infty}_{t}(\eta)$
of all smooth 
vector fields 
on $[0,t]$
along $\eta$
which are zero at the end points $e$ and $\eta(t)$.
Then,
we write $H_{t}:=H_{t}(\eta)$
for the completion of $C^{\infty}_{t}$
by the norm induced from the inner product
$$
\langle
X,
Y
\rangle_{H_{t}}
:=
\int_{0}^{t}
(\widetilde{\nabla}_{\dot{\eta}}X(t'),
\widetilde{\nabla}_{\dot{\eta}}Y(t'))
dt',
$$
where
$X,Y\in C_{t}^{\infty}$.
Extending elements of $H_{t}$
by zero on $[t,1]$ along $\eta$,
we regard $H_{t}\subset H_{t'}$
for $0\leq t\leq t'\leq 1$.
Then,
the second variation
$E''(\eta)_{0}^{t}(\cdot,\cdot)$
of the length function $E(\eta)_{0}^{t}$
(see \eqref{Eint})
defines a bounded symmetric bilinear form
on $H_{t}$,
which is given by
\begin{eqnarray*}
E''(\eta)^{t}_{0}(X,Y)
=
\int_{0}^{t}
\{
(\widetilde{\nabla}_{\dot{\eta}}X(t'),
\widetilde{\nabla}_{\dot{\eta}}Y(t'))
-
(
\widetilde{R}_{\eta}
(X(t'), \dot{\eta}
)\dot{\eta},
Y(t'))
\}dt'.
\end{eqnarray*}
Recall that the index of the form 
$E''(\eta)^{t}_{0}(\cdot,\cdot)$
is the dimension of the largest subspace of $H_{t}$
on which
$E''(\eta)^{t}_{0}(\cdot,\cdot)$ is negative definite.
We note that
the subset of $H_{t}$,
on which the form 
$E''(\eta)^{t}_{0}(\cdot,\cdot)$
is negative,
is not closed under addition.
Therefore,
even if the index is finite,
the subset can contain 
infinitely many linear independent vectors.
\begin{fact}[{\cite[Theorem 8.2]{MP}}]
\label{It}
Let
$\eta(t)$
($0\leq t \leq 1$)
be a geodesic from the identify $e$ to
$\eta(1)$
in the group of volume preserving diffeomorphisms
${\mathcal D}^{s}_{\mu}(M)$ of a surface without boundary.
Then,
the index of
$E''(\eta)^{1}_{0}(\cdot,\cdot)$ 
is finite and
equal to the number of conjugate points to
$e$
along $\eta$ each counted with multiplicity.
\end{fact}

It could seem that
Corollary \ref{negative}
contradicts to this fact
as we explained in the beginning of this section.
The proposition given below in this Appendix states that
this is not a contradiction:
\begin{prop}
For any $1<k<l$,
we have
$$
\int_{0}^{t_{V,W,l}}
(
\widetilde{\nabla}_{\dot{\eta}} \widetilde{W}^{k},
\widetilde{\nabla}_{\dot{\eta}} \widetilde{W}^{l})
dt
> 0.
$$
In other words,
$\widetilde{W}^{k}$
and
$\widetilde{W}^{l}$
are not orthogonal.
\end{prop}
If $\{\widetilde{W}^{k}\}$ form  
an infinite-dimensional vector subspace of $H_{t}$,
there exist  $k_1>k_2>\cdots$ such that
$\int_{0}^{t_{V,W,k_i}}
(
\widetilde{\nabla}_{\dot{\eta}} \widetilde{W}^{k_i},
\widetilde{\nabla}_{\dot{\eta}} \widetilde{W}^{k_j})
dt=0$ ($i>j$).
Therefore this proposition means that
$\{\widetilde{W}^{k}\}$ does not form 
any infinite-dimensional vector subspace of $H_{t}$,
which solves the apparent paradox.
For the proof,
let us recall that
\begin{eqnarray*}
t_{V,W,k}:=\pi|W|\sqrt{\frac{k}{MC_{V,W}}},\quad
f_{V,W,k}(t):=\sin \left(
\frac{t}{|W|}
\sqrt{\frac{MC_{V,W}}{k}}\right),\\
{\widetilde W}^{k}_{\eta(t)}:=
f_{V,W,k}(t)
(W\circ \eta(t))\in
T_{\eta(t)}{\mathcal D}^{s}_{\mu}(M).
\end{eqnarray*}
\begin{proof}
By
the same calculation of the proof of Proposition \ref{formulanegative},
we have
\begin{eqnarray}
(
\widetilde{\nabla}_{\dot{\eta}} \widetilde{W}^{k},
\widetilde{\nabla}_{\dot{\eta}} \widetilde{W}^{l})
&=&
\dot{f}_{V,W,k}
\dot{f}_{V,W,l}
|W|^{2}
+
f_{V,W,k}
f_{V,W,l}
|P_{e}\nabla_{V}W|^{2}.
\nonumber
\end{eqnarray}
We note that
\begin{eqnarray*}
&&
\int_{0}^{t_{V,W,l}}
\dot{f}_{V,W,k}
\dot{f}_{V,W,l}dt
\\
&=&
\frac{MC_{V,W}}
{|W|^{2}\sqrt{kl}}
\int_{0}^{t_{V,W,k}}
\cos \left(
\frac{t}{|W|}
\sqrt{\frac{MC_{V,W}}{k}}\right)
\cos \left(
\frac{t}{|W|}
\sqrt{\frac{MC_{V,W}}{l}}\right)
dt.
\end{eqnarray*}
Applying Product-Sum identities
and calculating the integral,
we have
\begin{eqnarray*}
&=&
\frac{\sqrt{MC_{V,W}}}
{2|W|(\sqrt{l}+\sqrt{k})}
\sin \left(
\left(
1+\sqrt{\frac{k}{l}}
\right)
\pi
\right)
+
\frac{\sqrt{MC_{V,W}}}
{2|W|(\sqrt{l}-\sqrt{k})}
\sin \left(
\left(
1-\sqrt{\frac{k}{l}}
\right)
\pi
\right)
\\
&=&
\frac{-\sqrt{MC_{V,W}}}
{2|W|(\sqrt{l}+\sqrt{k})}
\sin \left(
\pi
\sqrt{\frac{k}{l}}
\right)
+
\frac{\sqrt{MC_{V,W}}}
{2|W|(\sqrt{l}-\sqrt{k})}
\sin \left(
\pi
\sqrt{\frac{k}{l}}
\right).
\end{eqnarray*}
Moreover,
\begin{eqnarray*}
\int_{0}^{t_{V,W,l}}
f_{V,W,k}
f_{V,W,l}dt
&=&
\int_{0}^{t_{V,W,k}}
\sin \left(
\frac{t}{|W|}
\sqrt{\frac{MC_{V,W}}{k}}\right)
\sin \left(
\frac{t}{|W|}
\sqrt{\frac{MC_{V,W}}{l}}\right)
dt
\end{eqnarray*}
By Product-Sum identities
and the equalities $\sin(x+\pi)=-\sin(x)$,
$\sin(x-\pi)=\sin(x)$,
we have
\begin{eqnarray*}
&=&
\frac{|W|\sqrt{kl}}{2\sqrt{MC_{V,W}}}
\left(
\frac{1}
{(\sqrt{l}+\sqrt{k})}
\sin \left(
\pi
\sqrt{\frac{k}{l}}
\right)
+
\frac{1}
{(\sqrt{l}-\sqrt{k})}
\sin \left(
\pi
\sqrt{\frac{k}{l}}
\right)
\right)
\end{eqnarray*}
Thus,
we have
\begin{eqnarray*}
&&
\int_{0}^{t_{V,W,l}}
(
\widetilde{\nabla}_{\dot{\eta}} \widetilde{W}^{k},
\widetilde{\nabla}_{\dot{\eta}} \widetilde{W}^{l})
dt
\\
&=&
|W|
\left(
\frac{
\sqrt{MC_{V,W}}
\sqrt{k}
}
{
l^{2}-k^{2}
}
+
\frac{|P_{e}\nabla_{V}W|^{2}l\sqrt{k}}
{\sqrt{MC_{V,W}}(l^{2}-k^{2})}
\right)
\sin
\left(
\pi
\sqrt{\frac{k}{l}}
\right)
\\
&=&
\frac{|W|\sqrt{k}\left(
MC_{V,W}+
|P_{e}\nabla_{V}W|^{2}l
\right)}
{\sqrt{MC_{V,W}}(l^{2}-k^{2})}
\sin
\left(
\pi
\sqrt{\frac{k}{l}}
\right)
>
0.
\end{eqnarray*}
This completes the proof.
\end{proof}

\vspace{0.5cm}
\noindent
{\bf Acknowledgments.}\ 
The authors would like to thank G. Misio{\l}ek
for his helpful and careful comments.
The authors also thank the anonymous referee for his/her careful reading of our manuscript and many insightful comments and suggestions. 
Research of TT  was partially supported by 
Foundation of Research Fellows, The Mathematical Society of Japan.
Research of TY  was partially supported by 
Grant-in-Aid for Young Scientists A (17H04825),
Grant-in-Aid for Scientific Research B (15H03621, 17H02860, 18H01136 and 18H01135).

\bibliographystyle{amsplain}

\end{document}